\documentclass[11pt,oneside,reqno]{article}
 \usepackage{amstext,amsmath,amssymb,amsthm} 
\usepackage{latexsym}
\usepackage{exscale}
\usepackage{color}

\usepackage[margin=1in]{geometry} 
\usepackage{amsmath,amsthm,amssymb, graphicx, multicol, array}

\usepackage{pgfplots}
\pgfplotsset{compat=1.6}

\pgfplotsset{soldot/.style={color=blue,only marks,mark=*}} \pgfplotsset{holdot/.style={color=blue,fill=white,only marks,mark=*}}

\RequirePackage[cp1251]{inputenc}
\RequirePackage{srcltx}

\numberwithin{equation}{section}
\RequirePackage{srcltx}

\RequirePackage[cp1251]{inputenc}

\numberwithin{equation}{section}

\usepackage{mathtools}

\DeclarePairedDelimiter\floor{\lfloor}{\rfloor}

\theoremstyle{plain}
\newtheorem{Th}{Theorem}[section]
\newtheorem{Lemma}[Th]{Lemma}

\newtheorem{Prop}[Th]{Proposition}

 \theoremstyle{definition}
\newtheorem{Def}[Th]{Definition}

\newtheorem{Rem}[Th]{Remark}
\newtheorem{?}[Th]{Problem}

\renewcommand\hat{\widehat}
\def\XXint#1#2#3{{\setbox0=\hbox{$#1{#2#3}{\int}$ }
\vcenter{\hbox{$#2#3$ }}\kern-.6\wd0}}

\newcounter{cnstcnt}

\title{On the number and size of holes in the growing ball of first-passage percolation}
\author{Michael Damron \thanks{Email: mdamron6@protonmail.com. The research of M. D. is supported by an NSF grant DMS-2054559 and an NSF CAREER award.} \\ \small{Georgia Tech}  \and Julian Gold \thanks{Email: julian.thomas.gold@gmail.com. The research of J. G. is supported by NSF Postdoctoral Research Fellowship DMS-1803622}\\ \small{Northwestern University} \and Wai-Kit Lam\thanks{Email: waikitlam@ntu.edu.tw. The research of W.-K. L. is supported by Ministry of Science and Technology in Taiwan Grant MOST110-2115-M002-012-MY3 and NTU New Faculty Founding Research Grant NTU-111L7452.} \\ \small{National Taiwan University} \and Xiao Shen \thanks{Email: xiao.shen@utah.edu.} \\ \small{University of Utah} }

\begin{document}
	
	\maketitle 
	\begin{abstract}
		First-passage percolation is a random growth model defined on $\mathbb{Z}^d$ using i.i.d.~nonnegative weights $(\tau_e)$ on the edges. Letting $T(x,y)$ be the distance between vertices $x$ and $y$ induced by the weights, we study the random ball of radius $t$ centered at the origin, $B(t) = \{x \in \mathbb{Z}^d : T(0,x) \leq t\}$. It is known that for all such $\tau_e$, the number of vertices (volume) of $B(t)$ is at least order $t^d$, and under mild conditions on $\tau_e$, this volume grows like a deterministic constant times $t^d$. Defining a hole in $B(t)$ to be a bounded component of the complement $B(t)^c$, we prove that if $\tau_e$ is not deterministic, then a.s., for all large $t$, $B(t)$ has at least $ct^{d-1}$ many holes, and the maximal volume of any hole is at least $c\log t$. Conditionally on the (unproved) uniform curvature assumption, we prove that a.s., for all large $t$, the number of holes is at most $(\log t)^C t^{d-1}$, and for $d=2$, no hole in $B(t)$ has volume larger than $(\log t)^C$. Without curvature, we show that no hole has volume larger than $Ct \log t$.
	\end{abstract}
	
\section{Introduction}


\subsection{Backgound and definitions}

In the '60s, Hammersley-Welsh introduced first-passage percolation (FPP) on the cubic lattice $\mathbb{Z}^d$ as model for fluid flow in a porous medium. FPP is now often viewed in other ways: as a random growth model, a particle system, or a random metric space; see \cite{ADH17,DRS16} for recent surveys. In addition to the usual questions, like passage time asymptotics, the geometry of geodesics, and concentration bounds, attention has recently been paid to the boundary of the growing set $B(t)$ \cite{B15,DHL17,L85} and its topological properties \cite{MRS20}. The purpose of the current paper is to continue some of these newer questions, addressing the number and size of holes in $B(t)$.

Consider $\mathbb{Z}^d$, the $d$-dimensional integer lattice with nearest-neighbor edges $\mathbb{E}^d$. Let $(\tau_e)_{e \in \mathbb{E}^d}$ be an i.i.d.~family of nonnegative random variables (the edge-weights) assigned to the edges. A path from a vertex $x$ to a vertex $y$ is an alternating sequence $x = x_0, e_0, x_1, e_1, \dots, x_{n-1},e_{n-1},x_n = y$ of vertices and edges such that $e_i = \{x_i,x_{i+1}\} \in \mathbb{E}^d$ for all $i = 0, \dots, n-1$. The passage time of a path $\gamma$ is
\[
T(\gamma) = \sum_{i=0}^{n-1} \tau_{e_i},
\]
and the first-passage time from $x$ to $y$ is
\[
T(x,y) = \inf_{\gamma: x \to y} T(\gamma),
\]
where the infimum is over all paths $\gamma$ from $x$ to $y$.

We study the random ``ball''
\[
B(t) = \{x \in \mathbb{Z}^d : T(0,x) \leq t\} \text{ for } t \geq 0.
\]
The shape theorem of FPP gives a type of law of large numbers for $B(t)$, and states that the rescaled set $B(t)/t$ converges to a deterministic limiting shape as $t \to \infty$. The usual assumptions are that
\begin{equation}\label{eq: percolation_assumption}
\mathbb{P}(\tau_e = 0) < p_c(d),
\end{equation}
where $p_c(d)$ is the critical value for $d$-dimensional Bernoulli bond percolation (a constant known to be in the open interval $(0,1)$ for $d \geq 2$ and to be equal to $1/2$ for $d=2$), and $\mathbb{E}\min\{\tau_1^d, \dots, \tau_{2d}^d\}<\infty$, where the $\tau_i$ are i.i.d.~copies of $\tau_e$. Under these conditions \cite[Thm.~1.7]{aspects}, there exists a nonrandom convex set $\mathcal{B}$ which is invariant under permutations of the coordinates and under reflections in the coordinate hyperplanes, has nonempty interior, and which is compact, such that for all $\epsilon>0$,
\begin{equation}\label{eq: shape_theorem}
\mathbb{P}\left( (1-\epsilon)\mathcal{B} \subset \frac{1}{t} \widetilde{B}(t) \subset (1+\epsilon)\mathcal{B} \text{ for all large }t\right) = 1.
\end{equation}
Here, $\widetilde{B}(t)$ is the sum set $\{x+y : x \in B(t), y \in [0,1)^d\}$. This $\mathcal{B}$ is the unit ball of a norm $g$ on $\mathbb{R}^d$:
\[
\mathcal{B} = \{z \in \mathbb{R}^d : g(z) \leq 1\}.
\]
Hence, in the limit, the set $B(t)$ has no holes, but holes may be present for finite $t$.

Only assuming \eqref{eq: percolation_assumption}, Kesten's lemma \cite[Lem.~5.8]{aspects} implies that $B(t)$ is a.s.~finite for each $t \geq 0$ and so the complement $B(t)^c$ is a union of finitely many connected components. All but one of these components is finite. We then define the number of ``holes'' in $B(t)$ as
\[
N(t) = \text{number of finite connected components of }B(t)^c
\]
and the volume of the largest hole as
\[
M(t) = \max\left\{ \#S : S \text{ is a finite connected component of }B(t)^c\right\}.
\]
If $\tau_e$ is deterministic, then $N(t) = 0$ for all $t$, so we will assume
\begin{equation}\label{eq: non_trivial}
\text{the distribution of }\tau_e \text{ is non-trivial}.
\end{equation}
In other words, the support of the distribution of $\tau_e$ contains at least two points.

\subsection{Main results}	
	
Our results give upper and lower bounds on $N(t)$ and $M(t)$ under some conditions on the weights $(\tau_e)$. First are the lower bounds.

\begin{Th}\label{thm: lower_bounds}
Suppose \eqref{eq: percolation_assumption} and \eqref{eq: non_trivial} hold.
\begin{enumerate}
\item There exists $c>0$ such that 
\[
\mathbb{P}\left( M(t) \geq c \log t \text{ for all large }t \right) = 1.
\]
\item There exists $c>0$ such that
\[
\mathbb{P}\left( N(t) \geq c t^{d-1} \text{ for all large }t \right) = 1.
\]
\end{enumerate}
\end{Th}
\noindent
The proof of Thm.~\ref{thm: lower_bounds}
appears in Section~\ref{sec: lower_bounds}. Close inspection of the proof reveals that a.s., for all large $t$, the number of holes of $B(t)$ of volume at least $c \log t$ is at least $t^{d-1-\alpha}$ for some $\alpha$ which satisfies $\alpha(c) \to 0$ as $c \to 0$. 

The authors of \cite{MRS20} study the Betti numbers associated with the growing set in the Eden model, a simple model for cell growth. Their results give asymptotics for these numbers and, in particular, show that with high probability, the number of holes at time $t$ is the same order as the perimeter, which is at least $t^{d-1}$. The same bound therefore holds for a site-FPP model with exponential weights, because it is equivalent, through the memoryless property of exponentials, to the Eden model. Our proof of item 2 of Theorem~\ref{thm: lower_bounds} has a similar structure to theirs. For large $t$, we condition on $B(t)$ and find order $t^{d-1}$ many disjoint sets in $B(t)^c$ near the boundary of $B(t)$. Each such set has a positive probability to contain a special configuration that will develop into a hole in $B(t+C)$ for a constant $C>0$. Because we cannot use the memoryless property, finding and constructing these holes is more complicated. First, if the weights are bounded, we cannot just create a hole by increasing the weights of the $2d$ edges incident to a particular vertex in $B(t)^c$. Instead, in step 1 of the proof, we must define a more detailed high-weight event that ensures the existence of holes. Second, if the weights are unbounded, high-weight boundary edges may prevent $B(t+C)$ from enveloping our high-weight configurations outside $B(t)$ in constant time. We must therefore show in step 2 that for large $t$, the boundary of $B(t)$ contains many sections of low-weight edges that are near large areas in $B(t)^c$. 

\begin{Rem}
Holes in $B(t)$ were also previously studied in the proof of lower bounds on the size of the edge boundary of $B(t)$ in \cite[Thm.~1.3]{DHL17}. Their argument involves constructing order $t^d(1-F_Y(t))$ many unit-size holes in $B(t)$, where $F_Y$ is the distribution function of $\min\{\tau_1, \dots, \tau_{2d}\}$ and the $\tau_i$ are i.i.d.~copies of $\tau_e$. These holes arise from isolated vertices all of whose incident edges have high weight. When $\tau_e$ has a heavy tail, this number can be made arbitrarily close to $t^d$. The strategy from \cite{DHL17} does not obviously extend to lighter-tailed distributions, and the holes built in the proof of Thm.~\ref{thm: lower_bounds} above arise instead from large regions of slightly large edge-weights.
\end{Rem}

\begin{Rem}
As mentioned, if we remove assumption \eqref{eq: non_trivial}, we obtain $N(t) = 0$ for all $t$. Regarding assumption \eqref{eq: percolation_assumption}, if $F(0) > p_c$ but \eqref{eq: non_trivial} holds (that is, $\tau_e$ is not identically zero), then there is an infinite component of zero-weight edges a.s., and $\#B(t) = \infty$ for all large $t$. In addition, we have $N(t) = M(t) = \infty$ for all such $t$ a.s. On the other hand, the situation when $F(0) = p_c$ is more complicated because the growth rate of $B(t)$ can depend on the distribution of $\tau_e$ \cite{DLW17}. For some $\tau_e$, we still have $\#B(t) = \infty$ for all large $t$ (and so $N(t) = M(t) = \infty$), but for others, $\#B(t) <\infty$ for all $t$ a.s. Our proof of Thm.~\ref{thm: lower_bounds} can be used for $d=2$ to give a (probably nonoptimal) lower bound for $N(t)$ in terms of the growth rate of $B(t)$. For $d \geq 3$, there is not currently a simple condition on $\tau_e$ to determine if $\#B(t) = \infty$ for finite $t$. For these reasons, we leave this case for a future study.
\end{Rem}

Turning to upper bounds, each bounded component of $B(t)^c$ contributes at least one edge to the edge boundary of $B(t)$
\[
\partial_eB(t) = \{\{x,y\} : x \in B(t), y \notin B(t)\}.
\]
Therefore $N(t) \leq \#\partial_e B(t)$, and any upper bound for the size of the edge boundary holds also for $N(t)$. In \cite{DHL17}, Damron-Hanson-Lam gave some such inequalities, proving in particular that if $Y$ is the minimum of $2d$ many i.i.d.~edge-weights, then $\#\partial_e B(t)$ is at most order $t^{d-1} \mathbb{E}\min\{Y,t\}$ for ``most'' times (see \cite[Thm.~1.2]{DHL17}). This gives a weak complement to the inequality in item 2 of Thm.~\ref{thm: lower_bounds} when $\mathbb{E}Y<\infty$. We focus instead on a different result of \cite{DHL17} which involves the ``uniform curvature condition'' of Newman. This condition is unproved, but believed to be true for distributions of $\tau_e$ that are, say, continuous; see \cite[Sec.~2.8]{ADH17} for more details. 
\begin{Def}\label{def: uniform_curvature}
We say that the limit shape $\mathcal{B}$ satisfies the uniform curvature condition if there exist constants $c>0, \eta > 0$ such that for all $z_1,z_2 \in \partial \mathcal{B}$ and $z = (1-\lambda)z_1 + \lambda z_2$ with $\lambda \in [0,1]$,
\[
1-g(z) \geq c\min\{g(z-z_1),g(z-z_2)\}^\eta,
\]
where $g$ is the norm associated to $\mathcal{B}$.
\end{Def}
This condition is typically used in concert with an exponential moment condition:
\begin{equation}\label{eq: exponential_moments}
\mathbb{E}e^{\alpha \tau_e} < \infty \text{ for some } \alpha>0,
\end{equation}
but it is possible to define $\mathcal{B}$ and therefore uniform curvature without any moment condition on $\tau_e$.

As a consequence of the bound on $\#\partial_e B(t)$ from \cite[Thm.~1.5]{DHL17}, we immediately obtain the following.
\begin{Prop}
Suppose \eqref{eq: percolation_assumption} and \eqref{eq: exponential_moments} hold, and assume the uniform curvature condition for $\mathcal{B}$. There exists $C>0$ such that
\[
\mathbb{P}\left( N(t) \leq (\log t)^C t^{d-1} \text{ for all large }t\right) = 1.
\]
\end{Prop}
This result does not directly imply a good upper bound on the maximal hole size $M(t)$. For that, we give the following result in two dimensions.
\begin{Th}\label{thm: curvature_upper_bound}
Let $d=2$. Suppose \eqref{eq: percolation_assumption} and \eqref{eq: exponential_moments} hold, and assume the uniform curvature condition for $\mathcal{B}$. There exists $C>0$ such that
\[
\mathbb{P}\left( M(t) \leq (\log t)^C \text{ for all large }t\right) = 1.
\]
\end{Th}
The proof of Thm.~\ref{thm: curvature_upper_bound} is in Section~\ref{sec: curvature_upper_bound}. The argument bounds the diameter of a hole in both the radial direction and the lateral direction by $(\log t)^C$. The radial estimate (see ``The first case ...'' above \eqref{eq: masta_marinara}) is valid in general dimensions. To bound the diameter in the lateral direction (below \eqref{eq: last_case}), we must use planarity to trap a hole between two geodesics. This second part of the proof only works for $d=2$. It would be interesting to study the geometry of holes in more detail. Do the largest holes have larger diameter in the radial direction than in the lateral one? Is there an asymptotic shape for these holes?

Without the curvature assumption, the method of proof of Thm.~\ref{thm: curvature_upper_bound} still works in some form, and produces the following weaker result. It gives a bound on the diameter of a hole in both the radial and lateral direction of order $\sqrt{t \log t}$. Its proof is in Section~\ref{sec: no_curvature_upper_bound}.
\begin{Th}\label{thm: no_curvature_upper_bound}
Let $d=2$. Suppose \eqref{eq: percolation_assumption} and \eqref{eq: exponential_moments} hold. There exists $C>0$ such that
\[
\mathbb{P}\left( M(t) \leq Ct \log t \text{ for all large }t\right) = 1.
\]
\end{Th}

\subsection{Outline of the paper}

The rest of the paper consists of proofs of the main results. First, in Sec.~\ref{sec: lower_bounds}, we prove Thm.~\ref{thm: lower_bounds}.  The proof contains three steps. In step 1, we construct a high-weight event contained in an $\ell^1$-ball $\Lambda(n)$ that is used to create holes in $B(t)$. In step 2, we show how to find translates of $\Lambda(n)$ that are directly outside $B(t)$. In step 3, we put these tools together to prove that a.s., for all large $t$, many of the translates of $\Lambda(n)$ outside of $B(t)$ have high-weight configurations that turn into holes in $B(t)$ after a short time. In Sec.~\ref{sec: curvature_upper_bound}, we move to the proof of Thm.~\ref{sec: curvature_upper_bound}. The argument shows that a.s., for all large $t$, the largest hole in $B(t)$ must be contained in a sector of an annulus with volume of order $(\log t)^C$ (see Fig.~\ref{fig: fig_7}). Last, in Sec.~\ref{sec: no_curvature_upper_bound}, we show how to modify the proof from Sec.~\ref{sec: curvature_upper_bound} without the curvature assumption to prove Thm.~\ref{thm: no_curvature_upper_bound}.

\section{Proof of Thm.~\ref{thm: lower_bounds}}\label{sec: lower_bounds}

Throughout this section, we suppose that \eqref{eq: percolation_assumption} and \eqref{eq: non_trivial} hold. Therefore we can pick $a,b$ with $0 < a < b$ such that for every $\delta > 0$,
\begin{equation}\label{eq: a_b_delta}
\mathbb{P}(\tau_e \in [a-\delta,a]) > 0 \text{ and } \mathbb{P}(\tau_e \in [b,2b]) > 0
\end{equation}

\noindent
{\bf Step 1.} We first construct a high-weight event that ensures the existence of holes. For $n \geq 1$, let 
\[
\Lambda(n) = \{x \in \mathbb{Z}^d : \|x\|_1 \leq n\}
\]
and write $\mathbf{e}_i$ for the $i$-th coordinate vector. For $m_1,m_2,m_3 \geq 1$, define the region 
\[
R = R(m_1,m_2,m_3) = \left\{x \in \mathbb{Z}^d : -m_1 \leq x \cdot \mathbf{e}_1 \leq m_2, \sum_{i=2}^d |x \cdot \mathbf{e}_i| \leq m_3\right\},
\]
with interior boundary 
\[
\hat R = \hat R(m_1,m_2,m_3) = \{x \in R : \exists y \in \mathbb{Z}^d \setminus R \text{ with } \|x-y\|_1 = 1\}.
\]
Also define the discrete line segment 
\[
L = \{ k\mathbf{e}_1 : k = -n, \dots, -m_1\}.
\]
(See the left side of Fig.~\ref{fig: fig_1}.)

	\begin{figure}[h]
\hbox{\hspace{1.5cm}\includegraphics[width=0.9\textwidth, trim={0 -1cm 0cm -1cm}, clip]{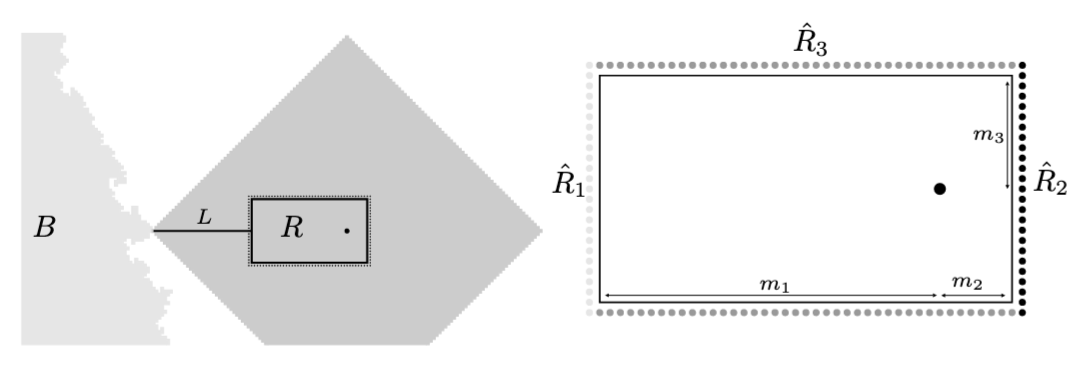}}
  \caption{On the left: the rectangle $R$ inscribed in the set $\Lambda(n)$, translated to touch the growing ball $B$ at a corner. On the right: the interior vertex boundary of $R$, defined as $\hat{R}$, is partitioned into $\hat{R}_1$, $\hat{R}_2$, and $\hat{R}_3$, indicated by the left (light grey), right (black), and top (dark grey) vertices. The origin is represented by the solid ball inside $R$.}
  \label{fig: fig_1}
\end{figure}

Given these geometric definitions, we now define our high-weight event. It is $E_n$, the event that
\begin{enumerate}
\item $\tau_e \in [a-\delta,a]$ for all $e = \{x,y\}$ with $x,y \in \hat R \cup L$, and
\item $\tau_e \in [b,2b]$ for all other $e = \{x,y\}$ with $x,y \in \Lambda(n)$.
\end{enumerate}
In step 3, we will use this event to create a hole in $B(t)$. The edges in item 1 allow one to enter $\Lambda(n)$ at $-n\mathbf{e}_1$, travel along $L$, and quickly encircle the high-weight region in $R$, where a hole can appear.

\begin{Lemma}\label{lem: barrel}
Let $a,b$ be as in \eqref{eq: a_b_delta} and $\epsilon < (b-a)/(2b+3a)$. If
\[
1 \leq m_2 \leq \epsilon m_3 \leq \epsilon^2 m_1 \leq \epsilon^3 n,
\]
then, for any $\delta>0$, on $E_n$,
\begin{equation}\label{eq: E_n_upper_bound}
T_{\Lambda(n)}(-n\mathbf{e}_1,y) \leq a(n+2m_3) + am_2 \text{ for all } y \in \hat R,
\end{equation}
and
\begin{equation}\label{eq: E_n_lower_bound}
T_{\Lambda(n)}(x,0) \geq (a-\delta)(n+2m_3)+bm_2 \text{ for all } x \in \mathbb{Z}^d \text{ with } \|x\|_1 = n,
\end{equation}
where $T_{\Lambda(n)}$ is the minimal passage time over paths whose vertices are in $\Lambda(n)$.
\end{Lemma}
\begin{proof}
Throughout the proof we will use the sides of $\hat R$:
\[
\hat R_1 = \{w \in \hat R : w \cdot \mathbf{e}_1 = -m_1\}, ~\hat R_2 = \{w \in \hat R : w \cdot \mathbf{e}_1 = m_2\},
\]
and
\[
\hat R_3 = \{w \in \hat R : -m_1 < w \cdot \mathbf{e}_1 < m_2\}.
\]
(See the right side of Fig.~\ref{fig: fig_1}.)

To show \eqref{eq: E_n_upper_bound}, let $y \in \hat R$; we will construct a path $\gamma$ from $-n\mathbf{e}_1$ to $y$ and estimate its passage time. By symmetry, we may assume that $y \cdot \mathbf{e}_i \geq 0$ for $i=2, \dots, d$. If $y \in \hat R_1 \cup \hat R_3$ then there is a $\gamma$ from $-n\mathbf{e}_1$ to $y$ with $\|-n\mathbf{e}_1 - y\|_1$ many edges all of which have both endpoints in $\hat R \cup L$. To build $\gamma$, start at $-n\mathbf{e}_1$ and move to $-m_1\mathbf{e}_1$ along $L$. If $y \in \hat R_1$, move to $y$ by increasing each $i$-th coordinate for $i=2, \dots, d$ in sequence. If $y \in \hat R_3$, move to $-m_1 \mathbf{e}_1 + \sum_{i=2}^d (y \cdot \mathbf{e}_i) \mathbf{e}_i$ by increasing each $i$-th coordinate for $i=2, \dots, d$ in sequence, and then move to $y$ by increasing the first coordinate. The path $\gamma$ as constructed has the desired properties, and  
\[
T(\gamma) \leq a\|-n\mathbf{e}_1 - y\|_1 \leq a(n+m_3 + m_2) \leq \text{RHS of } \eqref{eq: E_n_upper_bound}.
\]

If, instead, $y \in \hat R_2$, then we again move from $-n\mathbf{e}_1$ to $-m_1 \mathbf{e}_1$ along $L$, and then to the vertex 
\[
q = -m_1 \mathbf{e}_1 + \left(m_3 - \sum_{i=2}^d (y \cdot \mathbf{e}_i)\right) \mathbf{e}_2 + \sum_{i=2}^d (y \cdot \mathbf{e}_i) \mathbf{e}_i
\]
by increasing each $i$-th coordinate for $i=2, \dots, d$ in sequence. Then we move to $q + (m_1+m_2)\mathbf{e}_1$ by increasing the first coordinate, and finally decrease the second coordinate to reach $y$. This $\gamma$ as constructed has  
\[
(n-m_1) +m_3 + (m_1 + m_2) + (m_3 - \sum_{i=2}^d |y \cdot \mathbf{e}_i|) \leq n+2m_3 + m_2
\]
 many edges with weight $\leq a$, so we obtain \eqref{eq: E_n_upper_bound}.

For \eqref{eq: E_n_lower_bound}, we first show that if $x \in \mathbb{Z}^d$ has $\|x\|_1=n$, then
\begin{equation}\label{eq: hole_step_1}
T_{\Lambda(n)} (x,0) \geq T_{\Lambda(n)}(-n\mathbf{e}_1,0).
\end{equation}
To do this, let $u$ be the first intersection of any $T_{\Lambda(n)}$-optimal path from $x$ to $0$ with the set $\hat R \cup L$. Write $\gamma_1$ for the segment from $x$ to $u$ and $\gamma_2$ for the remaining segment. Then
\begin{align}
T_{\Lambda(n)}(x,0) - T_{\Lambda(n)}(-n \mathbf{e}_1,0) &\geq (T(\gamma_1) + T(\gamma_2)) - (T_{\Lambda(n)}(-n\mathbf{e}_1,u) + T(\gamma_2)) \nonumber \\
&= T(\gamma_1) - T_{\Lambda(n)}(-n\mathbf{e}_1,u). \label{eq: first_rotini}
\end{align}
Because $\gamma_1$ uses only edges with weight $\geq b$, $T(\gamma_1) \geq b \|x-u\|_1 \geq b(n-\|u\|_1)$. If $u \in L$ equals $-k\mathbf{e}_1$, then this is $b(n-k)$, but $T_{\Lambda(n)}(-n\mathbf{e}_1,u) \leq a(n-k)$, so \eqref{eq: first_rotini} is nonnegative. If, on the other hand, $u \in \hat R$, then $n-\|u\|_1 \geq n-(m_1+m_3)$ so by \eqref{eq: E_n_upper_bound},
\begin{align*}
T(\gamma_1) - T_{\Lambda(n)}(-n\mathbf{e}_1,u) \geq b(n-m_1-m_3) - a(n+2m_3+m_2) &\geq b(n-2m_1) - a(n+3m_1) \\
&\geq (b-a - 2b \epsilon - 3a\epsilon)n,
\end{align*}
which is $> 0$. This shows \eqref{eq: hole_step_1}.

	\begin{figure}[h]
\hbox{\hspace{1.5cm}\includegraphics[width=0.8\textwidth, trim={0 -1cm -1cm 1cm}, clip]{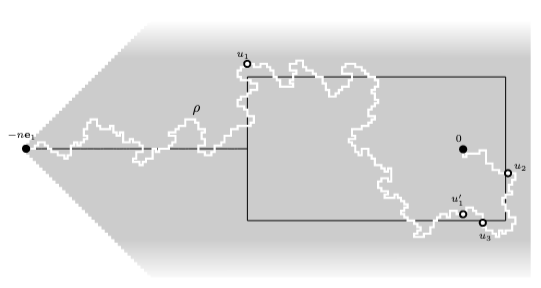}}
  \caption{Illustration of the last part of the proof of Lem.~\ref{lem: barrel}. The path depicted in white, $\rho$, is a $T_{\Lambda(n)}$-optimal path from $-n\mathbf{e}_1$ to 0.}
  \label{fig: fig_2}
\end{figure}

To prove \eqref{eq: E_n_lower_bound}, it now suffices by \eqref{eq: hole_step_1} to give the same lower bound for $T_{\Lambda(n)}(-n\mathbf{e}_1,0)$. Consider any $T_{\Lambda(n)}$-optimal path $\rho$ from $-n\mathbf{e}_1$ to $0$ and let $u_1$ be the last vertex of $\rho$ with $u_1 \cdot \mathbf{e}_1 = -m_1$. 
First, if $\rho$ contains no point in $\hat R_3$ after $u_1$, then let $u_1'$ be its first point after $u_1$ with $u_1' \cdot \mathbf{e}_1 = 0$. Then all edges on $\rho$ between $u_1$ and $u_1'$ have weight $\geq b$, so we obtain
\begin{equation}\label{eq: high_case_1}
T_{\Lambda(n)}(-n\mathbf{e}_1,0) = T(\rho) \geq (a-\delta)\|-n\mathbf{e}_1 - u_1\|_1 + b\|u_1-u_1'\|_1 \geq (a-\delta)(n-m_1) + bm_1.
\end{equation}
Otherwise, $\rho$ contains a point in $\hat R_3$ after $u_1$. Let $u_3$ be the last such point. If $\rho$ does not contain a point of $\hat R_2$ after $u_3$, then all edges on $\rho$ after $u_3$ have weight $\geq b$, and we obtain
\begin{align}
T_{\Lambda(n)}(-n\mathbf{e}_1,0) = T(\rho) &\geq (a-\delta)\|-n\mathbf{e}_1 - u_3\|_1 + b\|u_3\|_1 \nonumber \\
&=  (a-\delta) (n+u_3 \cdot \mathbf{e}_1) + (a-\delta) \sum_{i=2}^d |u_3 \cdot \mathbf{e}_i| + b\sum_{i=1}^d |u_3 \cdot \mathbf{e}_i| \nonumber \\
&\geq (a-\delta)(n+m_3) + bm_3. \label{eq: high_case_2}
\end{align}
Here we have used that $\sum_{i=2}^d |u_3 \cdot \mathbf{e}_i| = m_3$.

The last possibility, shown in Fig.~\ref{fig: fig_2}, is that $\rho$ contains a point of $\hat R_2$ after $u_3$; let $u_2$ be the last such one. Again, all edges on $\rho$ after $u_2$ must have weight $\geq b$, so $T_{\Lambda(n)}(-n\mathbf{e}_1,0)$ is at least
\begin{align}
& (a-\delta) \|-n\mathbf{e}_1 - u_3\|_1 + (a-\delta) \|u_3-u_2\|_1 + b\|u_2\|_1 \nonumber \\
=~& (a-\delta) (n+u_3 \cdot \mathbf{e}_1 + m_3) + (a-\delta) \left( (u_2-u_3) \cdot \mathbf{e}_1 + \sum_{i=2}^d |(u_3-u_2) \cdot \mathbf{e}_i|\right) + b\sum_{i=1}^d |u_2 \cdot \mathbf{e}_i| \nonumber \\
\geq~& (a-\delta)(n+m_3) + (a-\delta) \sum_{i=2}^d (|(u_3-u_2)\cdot \mathbf{e}_i| + |u_2 \cdot \mathbf{e}_i|) + (b+a-\delta)(u_2 \cdot \mathbf{e}_1) \nonumber \\
\geq~&(a-\delta)(n+2m_3) + bm_2. \label{eq: high_case_3}
\end{align}

We claim that among \eqref{eq: high_case_1}-\eqref{eq: high_case_3},
\begin{equation}\label{eq: minimal}
\text{the term in } \eqref{eq: high_case_3} \text{ is minimal.}
\end{equation}
Combining this fact with \eqref{eq: hole_step_1} will complete the proof of \eqref{eq: E_n_lower_bound}. To see why \eqref{eq: minimal} holds, we write the difference between the terms in \eqref{eq: high_case_1} and \eqref{eq: high_case_3} as
\begin{align*}
(a-\delta)(n-m_1) + bm_1 - (a-\delta)(n+2m_3) - bm_2 &\geq b(m_1-m_2) - a(m_1+2m_3) \\
&\geq (b-a-b\epsilon^2-2a\epsilon)m_1,
\end{align*}
which is $>0$. Also, the difference between the terms in \eqref{eq: high_case_2} and \eqref{eq: high_case_3} is
\begin{align*}
(a-\delta)(n+m_3) + bm_3 - (a-\delta)(n+2m_3)-bm_2 &= -(a-\delta)m_3 + b(m_3-m_2) \\
&\geq (b-\epsilon b - a)m_3,
\end{align*}
which is also $>0$. This completes the proof of \eqref{eq: minimal}.
\end{proof}

\bigskip
\noindent
{\bf Step 2.} Now that we have our high-weight event $E_n$ which takes place in $\Lambda(n)$, we describe a procedure to find translates of $\Lambda(n)$ that are directly outside the growing ball $B(t)$. These will house images of the event $E_n$, and will force holes in the ball at a time soon after $t$.

	\begin{figure}[t]
\hbox{\hspace{1cm}\includegraphics[width=0.9\textwidth, trim={0 -1cm 0cm -1cm}, clip]{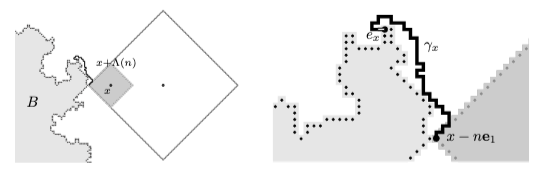}}
  \caption{On the left: a $(b',n)$-good vertex $x$, the center of the small diamond $x+\Lambda(n)$. The larger diamond and its center (corresponding to some $4ny$ for $y \in \partial^\infty B_n$, in notation introduced later) are not labeled, but are part of the statement in display \eqref{eq: placement_claim_1}. On the right: a close-up of the path $\gamma_x$ for the good vertex $x$. The initial vertex of $\gamma_x$ is $x-n\mathbf{e}_1$, and its terminal edge is $e_x$.}
  \label{fig: fig_3}
\end{figure}

Let $B$ be a finite connected set of vertices (like $B(t)$) and let $n \geq 1$ and $b' \geq 0$. We say that a vertex $x \in B^c$ is $(b',n)$-good for $B$ if
\begin{enumerate}
\item $x+\Lambda(n) \subset B^c$ but $x+\Lambda(n+1)$ intersects $B$, and
\item there exists a path $\gamma_x$ starting at a vertex of the form $x\pm n\mathbf{e}_j$ such that
\begin{enumerate}
\item $\gamma_x$ uses no vertices of either $x+\Lambda(n-1)$ or $B$,
\item some edge $e_x$ connects the final point of $\gamma_x$ to a vertex in $B$ and has $\tau_{e_x}\leq b'$, and
\item $\gamma_x$ has at most $\sqrt{n}$ many edges.
\end{enumerate}
\end{enumerate}

	\begin{figure}[h]
\hbox{\hspace{5cm}\includegraphics[width=0.4\textwidth, trim={0cm 0cm 0cm 0cm}, clip]{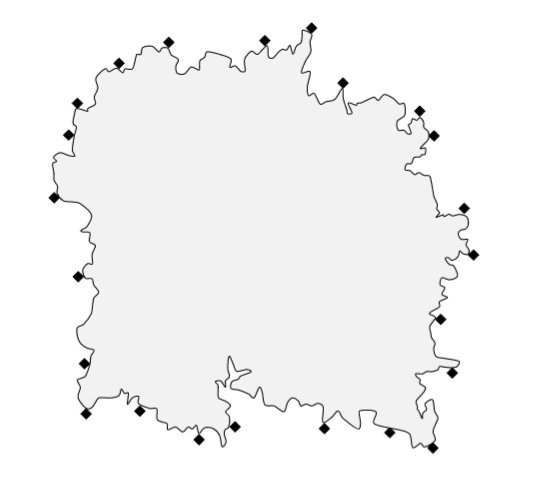}}
  \caption{Illustration of a $(b',n)$-good connected set $B$ of vertices. The translates $x+\Lambda(n)$, as $x$ ranges over the set $S(B)$ of vertices that are $(b',n)$-good for $B$, are shown in black.}
  \label{fig: fig_4}
\end{figure}

Fix a constant $c_1>0$. We say that $B$ is $(b',n)$-good if there is a set $S(B)$ of vertices $x$ that are $(b',n)$-good for $B$ such that
\begin{equation}\label{eq: B_b_n_good_def_new}
\text{any distinct }x,x' \in S(B) \text{ have } \|x-x'\|_1 \geq 4n
\end{equation}
and
\begin{equation}\label{eq: B_b_n_good_def}
\#S(B) \geq \frac{c_1}{n^{d-1}} \#B^{\frac{d-1}{d}}.
\end{equation}

Fig.~\ref{fig: fig_3} illustrates $(b',n)$-good vertices and Fig.~\ref{fig: fig_4} illustrates a $(b',n)$-good set $B$.

\begin{Prop}\label{prop: placement_prop}
There exist $b',c_1,C_2,c_3>0$ such that
\[
\mathbb{P}\left( \exists \text{ connected } B  \text{ with } 0 \in B, \#B = N \text{ and } B \text{ is not } (b',n)\text{-good}\right) \leq C_2 \left( \frac{N}{n}\right)^d \exp\left( - \frac{c_3}{n^{d-1}} N^{\frac{d-1}{d}}\right)
\]
for all large $n$ and for all $N \geq 1$.
\end{Prop}
\begin{proof}
Let $B$ be a connected set with $\#B= N$ and $0 \in B$. By taking $C_2$ large, we may assume that $N \geq (4n)^d$, so that
\begin{equation}\label{eq: B_conditions}
0 \in B \text{ and } B \text{ is not contained in }[-2n,2n-1]^d.
\end{equation} 
To verify that $B$ is $(b',n)$-good with high probability, we first consider vertices of the form $4ny$ which are directly outside of $B$. So, we cover $B$ with boxes to get
\[
B_n = \{z \in \mathbb{Z}^d : (4nz+[-2n,2n-1]^d) \cap B \neq \emptyset\},
\]
which is also a finite connected set. Using the notation 
\[
\partial^\infty \mathfrak{V} = \{y \in \mathfrak{V}^c : y \text{ is in the infinite component of }\mathfrak{V}^c, \exists z \in \mathfrak{V} \text{ with } \|y-z\|_\infty = 1\}
\]
for the exterior $\ast$-boundary of a finite connected $\mathfrak{V} \subset \mathbb{Z}^d$, we remark that $\partial^\infty B_n$ is connected \cite[Thm.~3]{T13}. For $v \in \mathbb{Z}^d$, define $\sigma_v = \max_w \tau_{\{v,w\}}$ to be the maximal weight over all edges with $v$ as an endpoint. For $y \in \mathbb{Z}^d$, let $F_y$ be the event that there exists a vertex self-avoiding path in $4ny+[-7n,7n-1]^d$ with $\lfloor \sqrt{n}\rfloor$ many edges and whose vertices $v$ satisfy $\sigma_v > b'$. In this first part of the proof, we show that 
\begin{equation}\label{eq: placement_claim_1}
\text{if }y \in \partial^\infty B_n\text{ and }F_y^c \text{ occurs, then some vertex in }4ny + [-6n,6n]^d \text{ is }(b',n)\text{-good for }B.
\end{equation}
(See Fig.~\ref{fig: fig_6}.)

To prove \eqref{eq: placement_claim_1}, suppose that $y \in \partial^\infty B_n$ and $F_y^c$ occurs. Because $y \notin B_n$, we have $4ny + [-2n,2n-1]^d \subset B^c$, but because there is a $z \in B_n$ with $\|z-y\|_\infty = 1$, we know $4ny + [-6n,6n-1]^d$ intersects $B$ at some point $y_0$. Choose $y_0$ so that $\|y_0-4ny\|_1 =  \min_{y' \in B}\|y'-4ny\|_1$. To select our point $x$ which will be $(b',n)$-good for $B$, we assume without loss in generality that $(y_0-4ny) \cdot \mathbf{e}_i \geq 0$ for all $i =1, \dots, d$, and that the first coordinate of $y_0-4ny$ is maximal. Because $\|y_0-4ny\|_\infty \geq 2n$, we find $(y_0-4ny) \cdot \mathbf{e}_1 \geq n+1$, and we define
\[
x = y_0 - (n+1)\mathbf{e}_1.
\]
Then $\|x-y_0\|_1 = n+1$ and so $x + \Lambda(n+1)$ intersects $B$ at the point $y_0 = x+(n+1)\mathbf{e}_1$. However, if $w \in x + \Lambda(n)$, then $\|w-4ny\|_1 \leq \|w-x\|_1 + \|x-4ny\|_1 \leq n + \|x-4ny\|_1$ and
\[
\|x-4ny\|_1 = \|y_0 - 4ny - (n+1)\mathbf{e}_1\|_1 = \|y_0-4ny\|_1 - (n+1),
\]
so $\|w-4ny\|_1 \leq \|y_0-4ny\|_1 - 1$, giving by minimality of $y_0$ that $w \notin B$. Therefore $x+\Lambda(n) \subset B^c$. Furthermore, because $(y_0 - 4ny) \cdot \mathbf{e}_1 = \|y_0 - 4ny\|_\infty \geq n+1$, we have $\|x-4ny\|_\infty = \|y_0-4ny - (n+1)\mathbf{e}_1\|_\infty \leq \|y_0 - 4ny\|_\infty \leq 6n$, so we conclude that $x \in 4ny + [-6n,6n]^d$. This shows that
\[
x \in 4ny + [-6n,6n]^d, ~x + \Lambda(n) \subset B^c, \text{ and } x+\Lambda(n+1) \text{ intersects }B \text{ at a point } y_0,
\]
where $y_0 = x + (n+1)\mathbf{e}_1$. Even without our assumptions on $(y_0-4ny) \cdot \mathbf{e}_i$, we obtain the same statement, but $y_0$ is then of the form $x \pm (n+1)\mathbf{e}_j$. This shows item 1 of the definition of $(b',n)$-good for the vertex $x$.

	\begin{figure}[h]
\hbox{\hspace{5cm}\includegraphics[width=0.4\textwidth, trim={0cm 0cm 0cm 0cm}, clip]{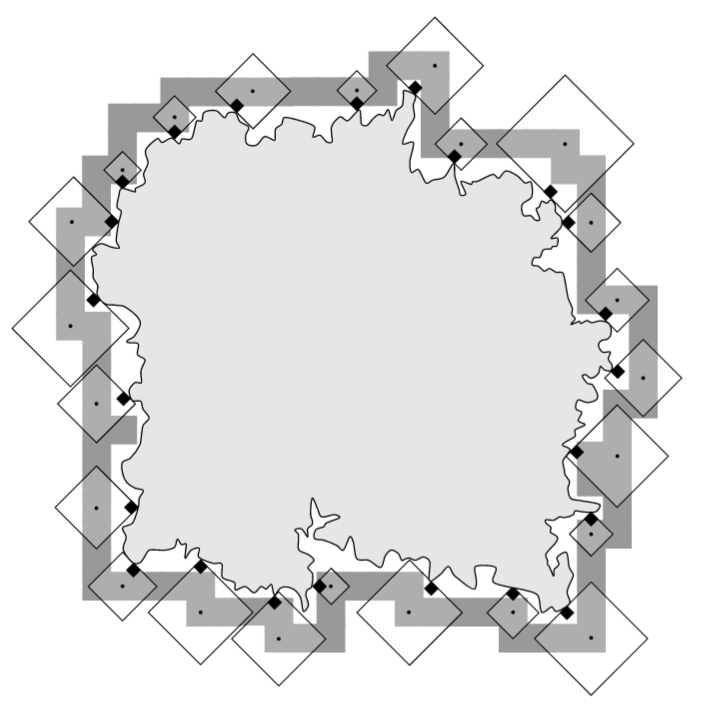}}
  \caption{Extracting a collection of $(b',n)$-good vertices (the centers of the black diamonds) surrounding a connected set $B$ of vertices, in light grey. The smaller diamonds are centered at vertices $4ny$ for $y$ in the set $\partial^\infty B_n$ (shown in grey) such that $F_y^c$ occurs.}
  \label{fig: fig_6}
\end{figure}

If the edge connecting $y_0$ to the unique vertex $w_0$ in $x+\Lambda(n)$ has weight $\leq b'$, then we can simply set $\gamma_x$ to be the path with no edges and a single vertex $w_0$. In general, though, we must find a nearby edge satisfying this weight constraint. We observe that $w_0$ is in the exterior $\ast$-boundary $\partial^\infty B$ of $B$. This is because it is adjacent to $y_0$, which is in $B$, but also can be connected to $4ny$ without touching $B$, and $y \in \partial^\infty B_n$. Display \eqref{eq: B_conditions} ensures that $B$ is not contained in $w_0+[-\sqrt{n},\sqrt{n} ]^d$. Neither is $\partial^\infty B$ if $n$ is large, and so we can select $w_0' \in \partial^\infty B$ which is not in $w_0+[-\sqrt{n},\sqrt{n} ]^d$. Because $\partial^\infty B$ is connected, there is a vertex self-avoiding path $\mathfrak{p}$ from $w_0$ to $w_0'$ in $\partial^\infty B$ and since $x+ \Lambda(n) \subset B^c$, the path $\mathfrak{p}$ cannot use any vertices of $x+ \Lambda(n-1)$. Let $\mathfrak{p}_x$ be the initial segment of $\mathfrak{p}$ consisting of the first $\lfloor \sqrt{n} \rfloor$ many edges and list the vertices of $\mathfrak{p}_x$ as $w_0, w_1, \dots, w_{\lfloor \sqrt{n}\rfloor}$. Each $w_i$ is an endpoint of an edge $f_i$ whose other endpoint is in $B$. If $\tau_{f_i} \leq b'$ for some $i$, we let $i_0$ be the first such $i$ and define $\gamma_x$ to be the initial segment of $\mathfrak{p}_x$ from $w_0$ to $w_{i_0}$. This $\gamma_x$ satisfies conditions (a)-(c) of the definition of $(b',n)$-good. If $i_0$ does not exist, then the entire path $\mathfrak{p}_x$ must have vertices with $\sigma_v > b'$, meaning that $F_y$ occurs. This shows \eqref{eq: placement_claim_1}.

Given \eqref{eq: placement_claim_1}, we can now return to the main proof. Let $B$ be connected with $\#B = N$ and such that \eqref{eq: B_conditions} holds. The set $B_n$ satisfies $(4n)^d \#B_n \geq \#B = N$, so the isoperimetric inequality implies 
\begin{equation}\label{eq: discrete_isoperimetric}
\#\partial^\infty B_n \geq \frac{c_4}{n^{d-1}}N^{\frac{d-1}{d}}.
\end{equation} 
Suppose that $B$ is not $(b',n)$-good. We claim that for some constant $C_5>0$, and $c_1$ from the definition of $(b',n)$-good,
\begin{equation}\label{eq: spacing_claim}
\sum_{y \in \partial^\infty B_n} \mathbf{1}_{F_y^c} \leq \frac{C_5c_1}{n^{d-1}} N^{\frac{d-1}{d}}.
\end{equation}
To see why, partition $\partial^\infty B_n$ into $C_{5} = C_{5}(d)$ many subsets $S_1, \dots, S_{C_{5}}$ such that if for a fixed $i$, we select distinct $y,y' \in S_i$, then $\|y-y'\|_\infty \geq 4$. If for some such $i$, $F_y^c$ and $F_{y'}^c$ both occur, then let $x,x'$ be the corresponding $(b',n)$-good points from \eqref{eq: placement_claim_1}. We have 
\[
\|x-x'\|_1 \geq \|x-x'\|_\infty \geq \|4ny-4ny'\|_\infty - \|4ny - x\|_\infty - \|4ny' - x'\|_\infty \geq 16n - 6n - 6n = 4n.
\]
The definition of $(b',n)$-good then implies $\sum_{y \in S_i}\mathbf{1}_{F_y^c} < (c_1/n^{d-1})N^{(d-1)/d}$ for $i=1, \dots, C_{5}$, and this gives \eqref{eq: spacing_claim}.

Since $B \subset [-N,N]^d$, we have $B_n \subset [-(N/n)-1,(N/n)+1]^d$ and so $\partial^\infty B_n \subset [-(N/n)-2, (N/n)+2]^d.$ Taking $c_4$ from \eqref{eq: discrete_isoperimetric}, if $A_{N,n}$ is the event that there exists a finite connected set $S \subset \mathbb{Z}^d$ such that
\[
\#S \geq \frac{c_4}{n^{d-1}} N^{\frac{d-1}{d}} \text{ and } S \subset \left[ - \frac{N}{n}-2, \frac{N}{n} + 2 \right]^d,
\]
but $\sum_{y \in S} \mathbf{1}_{F_y^c} \leq (c_1C_{5}/c_4)\#S$, then
\begin{equation}\label{eq: near_final_bound}
\mathbb{P}(\exists \text{ connected } B  \text{ with } 0 \in B, \#B = N \text{ and } B \text{ is not } (b',n)\text{-good}) \leq \mathbb{P}(A_{N,n}).
\end{equation}
Let $\mathcal{S}_k$ be the collection of connected $S \subset \mathbb{Z}^d$ such that $\#S = k$ and $S$ contains the origin. Using the bound $\#\mathcal{S}_k \leq (2de)^k \leq e^{C_6k}$ \cite{BBR10} for some $C_6>0$, we obtain for $\ell \geq 0$
\begin{align*}
& \#\text{ connected }S \subset \mathbb{Z}^d \cap [-\ell, \ell]^d \text{ with } \#S = k \\
\leq~&\sum_{v \in [-\ell,\ell]^d}  \# \text{ connected }S \subset \mathbb{Z}^d \text{ containing }v \text{ with } \#S= k \\
\leq~& (2\ell+1)^d e^{C_6k}.
\end{align*}
Applying this with $\ell = N/n + 2$, we see that
\begin{equation}\label{eq: almost_pasta}
\mathbb{P}(A_{N,n}) \leq \left( 2\frac{N}{n} + 5 \right)^d \sum_{k \geq \frac{c_4}{n^{d-1}}N^{\frac{d-1}{d}}} e^{C_6k} \max_{S \in \mathcal{S}_k} \mathbb{P}\left( \sum_{y \in S} \mathbf{1}_{F_y^c} \leq \frac{c_1C_{5}}{c_4}k\right).
\end{equation}
For a given $S \in \mathcal{S}_k$, the events $\mathbf{1}_{F_y^c}$ are not independent as $y$ ranges over $S$, but they are only finitely dependent. Therefore we can extract a subset of size at least $c_7 k$ such that as $y$ ranges over the subset, the events $\mathbf{1}_{F_y^c}$ are independent. This implies that
\[
\max_{S \in \mathcal{S}_k} \mathbb{P}\left( \sum_{y \in S} \mathbf{1}_{F_y^c} \leq \frac{c_1C_{5}}{c_4}k\right) \leq \mathbb{P}\left( \sum_{i=1}^{\lfloor c_7 k\rfloor} Z_i \leq \frac{c_1C_{5}}{c_4}k\right),
\]
where $Z_i$ are i.i.d.~ and have the same distribution as $\mathbf{1}_{F_0^c}$. The right side is bounded by
\[
\mathbb{P}\left( \sum_{i=1}^{\lfloor c_7 k\rfloor} (1-Z_i) \geq \lfloor c_7 k \rfloor - \frac{c_1C_{5}}{c_4}k\right) \leq 2^{c_7 k} \mathbb{P}(F_0)^{\lfloor c_7 k \rfloor - \frac{c_1C_{5}}{c_4}k},
\]
so we can return to \eqref{eq: almost_pasta} to state, for some $C_8>0$,
\begin{equation}\label{eq: near_final_bound_2}
\mathbb{P}(A_{N,n}) \leq \left( 2\frac{N}{n} + 5\right)^d \sum_{k \geq \frac{c_4}{n^{d-1}}N^{\frac{d-1}{d}}} e^{C_8 k} \mathbb{P}(F_0)^{\lfloor c_7 k \rfloor - \frac{c_1C_{5}}{c_4}k}.
\end{equation}

Last, we must estimate $\mathbb{P}(F_0)$. For a given vertex self-avoiding path $\gamma$ in $[-7n,7n-1]^d$ with $\lfloor \sqrt{n} \rfloor$ many edges,  the events $\{\sigma_v > b\}$ as $v$ ranges over the vertices of $\gamma$ are not independent, but they are finitely dependent. Again, we can find a subset of the vertices of size at least $c_9 \sqrt{n}$ such that as $v$ ranges over the subset, the events are independent. This gives
\[
\mathbb{P}(\text{for all } v \in \gamma, \sigma_v > b') \leq \mathbb{P}(\sigma_0>b)^{c_9\sqrt{n}-1} \leq (2d \mathbb{P}(\tau_e > b'))^{c_9\sqrt{n}-1}.
\]
The number of such paths $\gamma$ is at most $(14n)^d (2d)^{\sqrt{n}}$, so
\[
\mathbb{P}(F_0) \leq (14n)^d (2d)^{\sqrt{n}} (2d \mathbb{P}(\tau_e > b'))^{c_9\sqrt{n}-1}.
\]
Putting this in \eqref{eq: near_final_bound_2}, we find
\[
\mathbb{P}(A_{N,n}) \leq \left( 2\frac{N}{n} + 5\right)^d \sum_{k \geq \frac{c_4}{n^{d-1}}N^{\frac{d-1}{d}}} e^{C_8k} ((14n)^d (2d)^{\sqrt{n}} (2d \mathbb{P}(\tau_e > b'))^{c_9\sqrt{n}-1})^{\lfloor c_7 k\rfloor - \frac{c_1C_{5}}{c_4}k}.
\]
First choose $c_1$ so small that $\lfloor c_7 k\rfloor - c_1C_{5}k/c_4$ is at least $c_7k/2$. After this, we may choose $b'$ so large that the entire summand is at most $2^{-k}$. This produces the bound
\[
\mathbb{P}(A_{N,n}) \leq 2 \cdot \left( 2\frac{N}{n} + 5\right)^d 2^{- \frac{c_4}{n^{d-1}}N^{\frac{d-1}{d}}}.
\]
Combined with \eqref{eq: near_final_bound}, this implies the statement of Prop.~\ref{prop: placement_prop}.
\end{proof}

\bigskip
\noindent
{\bf Step 3.} In this step we use the tools from the previous steps to construct holes in $B(t)$. First, by \cite[Eq.~(3)]{CT16}, our assumption \eqref{eq: percolation_assumption} gives a $c_{10}>0$ such that
\begin{equation}\label{eq: volume_lower_bound}
\mathbb{P}(c_{10}t^d \leq \#B(t) < \infty \text{ for all large } t) = 1.
\end{equation}
To place the translates of $\Lambda(n)$ from step 2 around the set $B(t)$, we choose a size of one of the two forms
\begin{equation}\label{eq: f_t_choice}
n = n_t = C_{11} \in \mathbb{N} \text{ or }  \lfloor c_{12} (\log t)^{\frac{1}{d}} \rfloor.
\end{equation}
We fix the rest of our parameters as follows:
\begin{enumerate}
\item $a,b$ are as in \eqref{eq: a_b_delta} and $b'$ is from Prop.~\ref{prop: placement_prop},
\item let $\delta = \epsilon^4$, where $\epsilon < (b-a)/(2b+3a)$ (compare to Lem.~\ref{lem: barrel}) will be taken small in the proof of \eqref{eq: near_end} below,
\item $m_1 = \floor{\epsilon n}, m_3 = \floor{\epsilon m_1}, m_2 = \floor{\epsilon m_3}$ and set $\mathsf{L} = \floor{\epsilon m_2}$.
\end{enumerate}
If $C_{11}$ and $t$ are large with $c_{12}$ fixed, the $m_i$'s satisfy the constraints in Lem.~\ref{lem: barrel}. The parameter $\mathsf{L}$ will be a lower bound on the radius of a hole. Now we apply Prop.~\ref{prop: placement_prop} for
\begin{align}
&\mathbb{P}(\exists \text{ connected }B \text{ with } 0 \in B, c_{10}t^d \leq \#B <\infty \text{ and } B \text{ is not } (b',n)\text{-good}) \nonumber \\
\leq~&\sum_{N \geq c_{10}t^d} \mathbb{P}(\exists \text{ connected }B \text{ with } 0 \in B, \#B = N \text{ and } B \text{ is not } (b',n)\text{-good}) \nonumber \\
\leq~&C_2 \sum_{N \geq c_{10} t^d} \left( \frac{N}{n}\right)^d \exp\left( - \frac{c_3}{n^{d-1}} N^{\frac{d-1}{d}}\right).\label{eq: taco_suprema}
\end{align}
The application of Prop.~\ref{prop: placement_prop} requires that $n$ is large, and this holds for large $C_{11}$ and $t$ for fixed $c_{12}$. For either choice of $n$ from \eqref{eq: f_t_choice}, the expression in \eqref{eq: taco_suprema} is summable in $t$, so for any large $C_{11}$ and any fixed $c_{12}$,
\begin{equation}\label{eq: to_use_later}
\sum_{t \in \mathbb{N}} \mathbb{P}\left(c_{10}t^d \leq \#B(t) < \infty \text{ but } B(t) \text{ is not } (b',n_t)\text{-good}\right) < \infty.
\end{equation}

From the definition of $(b',n)$-good, we get boxes of the form $x+\Lambda(n)$ situated around our set $B(t)$, so now we must populate them with versions of the event $E_n$ from step 1. To do this properly, we need to decouple the variables inside $B(t)$ from those outside. For a given finite, connected $B$ containing the origin that is $(b',n)$ good, we may choose at least $(c_1/n^{d-1})\#B^{(d-1)/d}$ many vertices $x$ that are $(b',n)$-good for $B$ and distinct $x, x'$ satisfy inequality \eqref{eq: B_b_n_good_def_new}. These vertices come with edges $e_{x}$ and paths $\gamma_{x}$ as in the definition. The edges and paths are contained in the boxes $[-n-\sqrt{n}-1,n+\sqrt{n}+1]^d + x$ because of item 2(c) in the definition, and by \eqref{eq: B_b_n_good_def_new}, these boxes are disjoint for distinct $x,x'$. Enumerate the first
\[
r = \left\lceil \frac{c_1}{n^{d-1}} \#B^{\frac{d-1}{d}} \right\rceil
\]
many of these points in some deterministic way as $x_1, \dots, x_r$. All of the $x_i$, $\gamma_{x_i}$, and $e_{x_i}$ are random, so we must fix their values for a large $t \in \mathbb{N}$ as
\begin{align}
&\mathbb{P}\left( c_{10} t^d \leq \#B(t) < \infty \text{ and } B(t) \text{ is } (b',n)\text{-good}\right) \nonumber \\
=~& \sum_{B : c_{10}t^d \leq \#B < \infty} \sum_{(z_i,\pi_i,e_i)_{i=1}^r} \mathbb{P}\left(B(t) = B \text{ is } (b',n)\text{-good}, x_i = z_i, \gamma_{x_i} = \pi_i, e_{x_i} = e_i ~\forall i\right). \label{eq: time_to_decouple}
\end{align}
We observe that the event in the probability depends only on edges with at least one endpoint in $B$, so it is independent of the weights of edges with both endpoints outside of $B$.

For a given choice of $(z_i, \pi_i, e_i)_{i=1}^r$, and $B$, we define events $(A_i)_{i=1}^r$ by the following conditions. $A_i$ is the event that:
\begin{enumerate}
\item all edges $e$ of $\pi_i$ have $\tau_e \leq a$, and
\item the event $T_iE_{n}$ occurs.
\end{enumerate}
In item 2, $T_iE_n$ is a certain translation and rotation of the high-weight event $E_n$ from step 1. Precisely, the initial point of $\pi_i$ is one of the points of the form $z_i \pm n\mathbf{e}_j$, and we define $T_i$ to be an isometry of $\mathbb{R}^d$ that maps $\Lambda(n)$ to $z_i + \Lambda(n)$ and $-n\mathbf{e}_1$ to the initial point of $\pi_i$. Then $T_iE_n$ is the event that the image configuration $(\tau_{T_i^{-1}(e)})$ is in $E_n$. Not only does the definition of $E_n$ depend on $n$ from \eqref{eq: f_t_choice} and $a,b$, it also depends on $\delta = \epsilon^4$ from \eqref{eq: a_b_delta} and the numbers $m_1,m_2,m_3$. Regardless of the values of the $m_i$, since they are $\leq n$, there exists $C_{13}>0$ depending only on $a,b,\epsilon$ such that $\mathbb{P}(E_n) \geq e^{-C_{13} n^d}$. Using this in the definition of $A_i$, there exists $C_{14}>0$ also depending only on $a,b,\epsilon$ such that
\[
\mathbb{P}(A_i) \geq \exp\left( - C_{14}n^d\right) \text{ for all } i = 1, \dots, r.
\]
Because the $A_i$'s are independent, we may bound the family $(\mathbf{1}_{A_i})_{i=1}^r$ stochastically from below by a family $(W_i)_{i=1}^r$ of i.i.d.~Bernoulli variables with parameter $p = e^{-C_{14}n^d}$. By Hoeffding's bound for Bernoulli random variables, $\mathbb{P}\left(W_1 + \dots + W_r \leq r\frac{p}{2}\right) \leq \exp\left( - \frac{r}{2}p^2\right)$, and we obtain
\[
\mathbb{P}\left( \sum_{i=1}^r \mathbf{1}_{A_i} \leq \frac{r}{2} \exp\left( - C_{14}n^d\right)\right) \leq \exp\left( - \frac{r}{2} \exp\left( - 2C_{14} n^d\right)\right).
\]
For any large $C_{11}$ and small $c_{12}$, we have $r \geq c_{15}(t/n)^{d-1}$ for all large $t$, so 
\[
\frac{r}{2} \exp\left(-2C_{14}n^d\right) \geq t^{d-1}\exp\left(-C_{16}n^d\right) \geq c_{17}t^{d-\frac{3}{2}} \text{ for all large }t.
\]
This implies for any large $C_{11}$ and small $c_{12}$,
\[
\mathbb{P}\left( \sum_{i=1}^r \mathbf{1}_{A_i} \leq \exp\left( - C_{16}n^d\right)t^{d-1} \right) \leq \exp\left( - c_{17} t^{d-\frac{3}{2}} \right) \text{ for all large } t.
\]
Returning to the right side of \eqref{eq: time_to_decouple}, independence gives for any large $C_{11}$ and small $c_{12}$, 
\begin{align}
&\left( 1 - \exp\left( - c_{17} t^{d-\frac{3}{2}} \right)\right) \mathbb{P}\left( c_{10} t^d \leq \#B(t) < \infty \text{ and }B(t) \text{ is } (b',n)\text{-good}\right) \nonumber \\
\leq~& \sum_{B : c_{10}t^d \leq \#B < \infty} \sum_{(z_i,\pi_i,e_i)_{i=1}^r} \mathbb{P}\left(
\begin{array}{c}
B(t) = B \text{ is } (b',n)\text{-good}, x_i = z_i, \gamma_{x_i} = \pi_i, e_{x_i} = e_i ~\forall i, \\
\sum_{i=1}^r \mathbf{1}_{A_i} \geq \exp\left( - C_{16}n^d\right) t^{d-1}
\end{array}
\right) \label{eq: really_time_to_decouple}
\end{align}
for all large $t$.

We will now argue that there exists $\epsilon>0$ such that on the event on the right of \eqref{eq: really_time_to_decouple}, if $C_{11}$ is any large number and $c_{12}$ is any fixed number, then for all large $t$, and all $i$ such that $A_i$ occurs,
\begin{equation}\label{eq: near_end}
x_i + \Lambda(\mathsf{L}) \text{ is in a bounded component of }B(s)^c \text{ for all } s \in [t+\kappa, t+\kappa +\epsilon^4 n], 
\end{equation}
where
\begin{equation}\label{eq: s_constraint}
\kappa = \kappa_t = \epsilon^4 n + a(n+2\epsilon^2 n) + a \epsilon^3 n,
\end{equation}
and these components are distinct for distinct values of $i$. In this statement, as before, $n = n_t$, so that $\mathsf{L}$ (defined below \eqref{eq: f_t_choice}) and $\kappa$ are also functions of $t$ (not $s$). To prove this, pick an outcome in this event with $i$ such that $A_i$ occurs, and let $u_i$ be the endpoint of $e_i$ in $B$. Let $v_i$ be the endpoint of $\pi_i$ that is in $x_i + \Lambda(n)$. Let $y \in T_i \hat R$ (this is the corresponding image of the set $\hat R$ from Lem.~\ref{lem: barrel} inside $x+\Lambda(n)$). Because $u_i \in B(t)$ and $\pi_i$ has at most $\sqrt{n}$ many edges, we have
\begin{align}
T(0,y) \leq T(0,u_i) + \tau_{e_i} + T(\pi_i) + T(v_i,y) &\leq t + b' + a\sqrt{n}+ a(n+2m_3)+am_2 \nonumber \\
&\leq t + \kappa. \label{eq: y_upper_bound}
\end{align}
We have used \eqref{eq: E_n_upper_bound} to estimate $T(v_i,y)$ and used $b' + a \sqrt{n} \leq \epsilon^4 n$, which is valid for any $\epsilon$ and $c_{12}$ so long as $C_{11}$ and $t$ are large. On the other hand, if $z \in \mathbb{Z}^d$ has $\|z-x_i\|_1 \leq \mathsf{L} \leq \min\{m_i\}$, condition 2 of the definition of $E_n$ implies
\[
T(0,z) \geq T(0,x_i) - T(x_i,z) \geq T(0,x_i) - 2\mathsf{L}b.
\]
Let $\sigma$ be any path from $0$ to $x_i$, let $\sigma_1$ be the initial segment until its first vertex outside $B$, and let $\sigma_2$ be its terminal segment starting at the point at which it enters $x_i + \Lambda(n)$ for the last time. Then because $\sigma_1$ connects 0 to $B(t)^c$,
\begin{align*}
T(\sigma) \geq T(\sigma_1) + T(\sigma_2) &\geq t + \min_{x : \|x-x_i\|_1 = n} T_{T_i\Lambda(n)}(x,x_i) \\
&\geq t + (a-\delta)(n+2m_3) + bm_2.
\end{align*}
The last inequality follows from \eqref{eq: E_n_lower_bound}. Take the infimum over $\sigma$ to obtain
\begin{align*}
T(0,z) \geq T(0,x_i) - 2\mathsf{L}b &\geq t + (a-\delta)(n+2m_3) + bm_2 - 2\mathsf{L}b \\
&\geq t + (a-\epsilon^4)(n+ 2 \epsilon^2 n) + b \epsilon^3 n - (2b+1)\epsilon^4 n \\
&= t + \kappa + (b-a)\epsilon^3 n - \epsilon^4(n+2\epsilon^2 n) - (2b+2)\epsilon^4 n.
\end{align*}
Again we have assumed that $\epsilon$ is fixed, $c_{12}$ is fixed, and $C_{11}$ and $t$ are large to remove the floor function in the definition of the $m_i$'s. From the above, we can choose $\epsilon$ so small such that for any $c_{12}$, and for any large $C_{11}$,
\[
T(0,z) \geq t + \kappa + \epsilon^4 n \text{ for all large } t.
\]
This inequality and \eqref{eq: y_upper_bound} show that for any $s$ in the interval described in \eqref{eq: near_end}, the set $x_i + \Lambda(\mathsf{L})$ is in $B(s)^c$, but $T_i \hat R$ is in $B(s)$. This implies \eqref{eq: near_end}. Furthermore, the sets $x_i + \Lambda(n)$ are disjoint, so since the components described in \eqref{eq: near_end} are contained in these sets, they are distinct for distinct values of $i$.

Given \eqref{eq: near_end}, we can finish the proof. Any component listed in \eqref{eq: near_end} contains $x_i + [0,\mathsf{L}/d]^d$, so it has at least $(\mathsf{L}/d)^d$ many vertices. If we define
\[
Y_t = \min_{s \in [t+\kappa_t, t+\kappa_t + \epsilon^4 n_t]} \# \text{ bounded components of }B(s)^c \text{ with at least } \left( \frac{\mathsf{L}_t}{d} \right)^d \text{ many vertices},
\]
then we can continue from \eqref{eq: really_time_to_decouple} with our $\epsilon$ from \eqref{eq: near_end}, any large $C_{11}$ and any small $c_{12}$ to obtain
\begin{align}
&\left( 1 - \exp\left( - c_{17} t^{d-\frac{3}{2}} \right)\right) \mathbb{P}\left( c_{10} t^d \leq \#B(t) < \infty \text{ and }B(t) \text{ is } (b',n)\text{-good}\right) \nonumber \\
\leq~& \sum_{B : c_{10}t^d \leq \#B < \infty} \sum_{(z_i,\pi_i,e_i)_{i=1}^r} \mathbb{P}\left(
\begin{array}{c}
B(t) = B \text{ is } (b',n)\text{-good}, x_i = z_i, \gamma_{x_i} = \pi_i, e_{x_i} = e_i ~\forall i, \\
Y_t \geq \exp\left( - C_{16}n^d\right) t^{d-1}
\end{array}
\right) \nonumber \\
=~& \mathbb{P}\left(B(t) \text{ is } (b',n)\text{-good, } c_{10}t^d \leq \#B(t) <\infty,~ Y_t \geq \exp\left( - C_{16} n^d\right)t^{d-1}\right) \label{eq: pizza_pie}
\end{align}
for all large $t$. This implies for any large $C_{11}$ and any small $c_{12}$
\begin{align*}
&\sum_{t \in \mathbb{N}} \mathbb{P}\left(c_{10} t^d \leq \#B(t) < \infty \text{ and } Y_t < \exp\left( - C_{16}n_t^d\right)t^{d-1}\right) \\
\leq~& \sum_{t \in \mathbb{N}} \mathbb{P}\left(c_{10}t^d \leq \#B(t) < \infty \text{ but } B(t) \text{ is not } (b',n_t)\text{-good}\right) \\
+~& \sum_{t \in \mathbb{N}} \mathbb{P}\left(Y_t < \exp\left( - C_{16}n_t^d\right)t^{d-1} \mid B(t) \text{ is }(b',n_t)\text{-good}, c_{10} t^d \leq \#B(t) <\infty\right).
\end{align*}
The sum in the second line is finite by \eqref{eq: to_use_later}. By \eqref{eq: pizza_pie}, the summands of the third are bounded for large $t$ by the summands of $\sum_{t \in \mathbb{N}} \exp\left( - c_{17} t^{d-\frac{3}{2}} \right) < \infty.$ The Borel-Cantelli lemma combined with \eqref{eq: volume_lower_bound} therefore implies that for any large $C_{11}$ and any small $c_{12}$, a.s.,
\begin{equation}\label{eq: N_t_conclusion}
Y_t \geq \exp\left( - C_{16} n_t^d\right) t^{d-1} \text{ for all large }t \in \mathbb{N}.
\end{equation}

Last, we use \eqref{eq: N_t_conclusion} to prove Thm.~\ref{thm: lower_bounds}. First take $n_t = C_{11}$. Then $\kappa_t = (\epsilon^4 + a + 2\epsilon^2 + a \epsilon^3)C_{11}$, and so the interval $I_t = [t+\kappa_t, t + \kappa_t + \epsilon^4 n_t]$ satisfies $I_t \cap I_{t+1} \neq \emptyset$ for all $t \geq 1$ so long as $C_{11}$ is large. Therefore \eqref{eq: N_t_conclusion} gives that a.s.~$B(t)^c$ has at least $\exp\left( - C_{16} C_{11}^d\right)t^{d-1}$ many bounded components for all large $t$. This proves item 2 of Thm.~\ref{thm: lower_bounds}. If we take $n_t = \floor{c_{12} (\log t)^{1/d}}$, then the intervals $I_t$ and $I_{t+1}$ also intersect for large $t$, if $c_{12}$ is fixed. For small $c_{12}$, we have $Y_t \geq t^{d-3/2}$ for all large $t$ so, in particular, $Y_t>0$. This gives that a.s., for all large $s$, the maximum hole size $M(s)$ is at least equal to $(\mathsf{L}_t/d)^d$, where $t$ is any number such that $s \in I_t$. If $t$ is large and $c_{12}$ is fixed, then this $t$ satisfies $t \geq s/2$, so we obtain
\[
\text{a.s. }, M(s) \geq \left( \frac{\mathsf{L}_{\frac{s}{2}}}{d}\right)^d \text{ for all large } s.
\]
This implies item 1 of Thm.~\ref{thm: lower_bounds} and completes the proof.

\section{Proof of Thm.~\ref{thm: curvature_upper_bound}}\label{sec: curvature_upper_bound}

In this section, we assume \eqref{eq: percolation_assumption}, \eqref{eq: exponential_moments}, and the uniform curvature condition. We first describe the idea of the proof. Let $t$ be large and let $x_0$, if it exists, be any vertex in the largest bounded component $\mathsf{C}$ of $B(t)^c$ with maximal Euclidean norm $\|x_0\|_2$. Let $\theta(v,w)$ be the angle (in $(-\pi,\pi]$) between $v,w \in \mathbb{R}^2$ and define the sector portion
\begin{equation}\label{eq: S_x_0_def}
S_{x_0} = \left\{v \in \mathbb{R}^2 : |\theta(v,x_0)| \leq J_{x_0}, ~1-K_{x_0} \leq \frac{\|v\|_2}{\|x_0\|_2} \leq 1 \right\},
\end{equation}
where 
\begin{equation}\label{eq: K_x_0_def}
J_{x_0} = \frac{(\log \|x_0\|_2)^{C_{18}-3}}{\|x_0\|_2},~K_{x_0} = \frac{(\log \|x_0\|_2)^{C_{18}} }{\|x_0\|_2},
\end{equation} 
and $C_{18} > 3$ is a large constant to be chosen later; see Fig.~\ref{fig: fig_7}. The component $\mathsf{C}$ containing $x_0$ is connected, and by extremality of $x_0$, it cannot cross the far side of $S_{x_0}$. Once we show that it cannot cross the left, right, and near sides, then we can deduce that $\mathsf{C} \subset S_{x_0}$. Because 
\begin{equation}\label{eq: S_x_0_volume}
S_{x_0} \text{ contains at most }C_{19} (\log \|x_0\|_2)^{2C_{18}} \text{ many vertices},
\end{equation}
and $\|x_0\|_2$ must be of order $t$ to be in a bounded component of $B(t)^c$, we conclude the result.

	\begin{figure}[h]
\hbox{\hspace{2.5cm}\includegraphics[width=0.6\textwidth, trim={0cm 0cm 0cm 0cm}, clip]{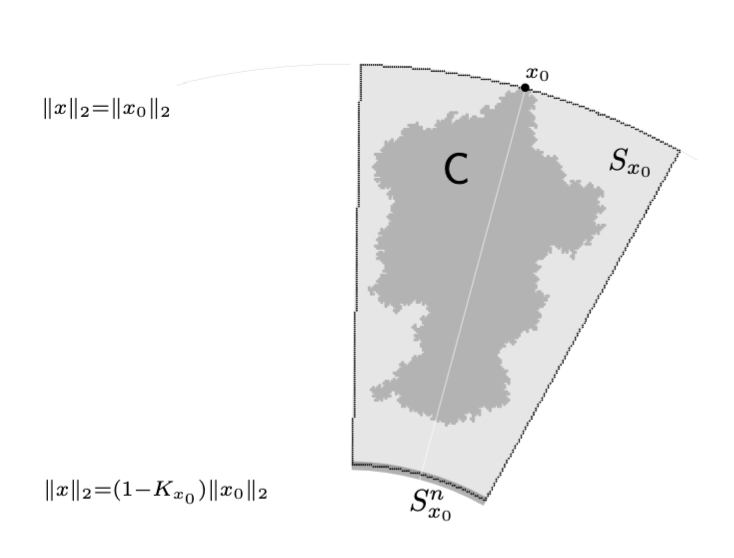}}
  \caption{The set $\mathsf{C}$, depicted above as the darker shaded region, is the largest hole in the ball $B(t_0)$. The set $S_{x_0}$ is a sector (lighter shaded region) centered on the vertex $x_0$ with maximal Euclidean norm among all those in $\mathsf{C}$. The boundary segment of $S_{x_0}$ nearest to the origin, $S_{x_0}^n$, is also shaded.}
  \label{fig: fig_7}
\end{figure}


To start the proof, we let $s>0$ and define the events
\[
E_1(s) = \left\{ \frac{1}{2} \mathcal{B} \subset \frac{1}{t} \widetilde{B}(t) \subset 2 \mathcal{B} \text{ for all } t \geq s \right\}
\]
and
\[
E_2(s) = \left\{ \tau_e \leq C_{20} \log t \text{ for all } e \text{ with an endpoint in }3t\mathcal{B} \text{ and all } t \geq s\right\}.
\]
In the definition of $E_1(s)$, we recall the notation $\widetilde{B}(t) = B(t) + [0,1)^d$ from the introduction. We have
\begin{align*}
\mathbb{P}(M(t) \geq (\log t)^{3C_{18}} \text{ for some }t \geq s) &\leq \mathbb{P}(E_1(s)^c) + \mathbb{P}(E_2(s)^c) \\
&+ \mathbb{P}\left(E_1(s) \cap E_2(s) \cap \{M(t) \geq (\log t)^{3C_{18}} \text{ for some } t \geq s\}\right).
\end{align*}
By the shape theorem in \eqref{eq: shape_theorem}, $\mathbb{P}(E_1(s)^c) \to 0$ as $s \to \infty$. To estimate $\mathbb{P}(E_2(s)^c)$, we write $\mathbb{P}(\tau_e > C_{20} \log n) \leq \mathbb{E}e^{\alpha \tau_e} / e^{\alpha C_{20}\log n}$ for the $\alpha$ in \eqref{eq: exponential_moments}, so
\[
\mathbb{P}(\tau_e > C_{20} \log n \text{ for some } e \text{ with an endpoint in }3n\mathcal{B}) \leq C_{21}n^2 n^{-C_{20} \alpha}.
\]
By a union bound, $\mathbb{P}(\tau_e > C_{20} \log n \text{ for some } e \text{ with an endpoint in }3n \mathcal{B} \text{ and some } n \geq N) \to 0$ as $N \to \infty$ if we choose $C_{20} > 4\alpha$. By increasing $C_{20}$ further, this implies that $\mathbb{P}(E_2(s)^c) \to 0$ as $s \to \infty$.

From the above arguments, we obtain
\begin{align}
&\lim_{s \to \infty} \mathbb{P}(M(t) \geq (\log t)^{3C_{18}} \text{ for some }t \geq s) \nonumber \\
=~& \lim_{s \to \infty} \mathbb{P}\left(E_1(s) \cap E_2(s) \cap \{M(t) \geq (\log t)^{3C_{18}} \text{ for some } t \geq s\}\right). \label{eq: new_start}
\end{align}
To show the limit in \eqref{eq: new_start} is zero, we use the sector construction from the proof idea above. Fix an outcome in the event in the probability in \eqref{eq: new_start} and let $t_0 \geq s$ be any value of $t$ for which $M(t) \geq (\log t)^{3C_{18}}$. Choose $x_0$ as any vertex with maximal Euclidean norm in a bounded component $\mathsf{C}$ of $B(t_0)^c$ with the largest number of vertices, and let $S_{x_0},J_{x_0},K_{x_0}$ be as in \eqref{eq: S_x_0_def} and \eqref{eq: K_x_0_def}. We first argue that for large $s$,
\begin{equation}\label{eq: containment}
\mathsf{C} \text{ contains a vertex in }S_{x_0}^c.
\end{equation}
To do this, we note that there exists $c_{22}>0$ such that 
\begin{equation}\label{eq: balls_def}
[0,c_{22}] \subset \{\|w\|_2 : w \in \mathcal{B}\} \subset \left[0,c_{22}^{-1}\right].
\end{equation}
Because $x_0$ is adjacent to $B(t_0)$ and $E_1(s)$ occurs, we have 
\[
\|x_0\|_2 \leq 1+ \max_{x \in B(t_0)} \|x\|_2  \leq 1+ 2 t_0 \max_{x \in \mathcal{B}} \|x\|_2 \leq 1 + 2c_{22}^{-1} t_0.
\]
As $x_0 \in B(t_0)^c$, we have $\|x_0\|_2 \geq (t_0/2)\max_{x \in \mathcal{B}} \|x\|_2 \geq c_{22}t_0/2$. 
In summary,
\begin{equation}\label{eq: t_0_relation}
\frac{c_{22}}{2}t_0 \leq \|x_0\|_2 \leq 1 + \frac{2}{c_{22}} t_0.
\end{equation}
Now for a contradiction, assume that $\mathsf{C} \subset S_{x_0}$. Then from \eqref{eq: S_x_0_volume}, we get 
\begin{equation}\label{eq: M_t_0_pre_bound}
M(t_0) \leq C_{19} (\log \|x_0\|_2)^{2C_{18}}.
\end{equation}
Combining this with \eqref{eq: t_0_relation}, we obtain
\[
M(t_0) \leq C_{19} (\log (1+2c_{22}^{-1}t_0))^{2C_{18}}.
\]
This contradicts $M(t_0) \geq (\log t_0)^{3C_{18}}$ for large $s$ because $t_0 \geq s$, and shows \eqref{eq: containment}.

We have now shown that for our outcome in the probability in \eqref{eq: new_start}, \eqref{eq: containment} holds. Let $\gamma$ be a path contained in $\mathsf{C}$ starting at $x_0$ that ends at a vertex outside of $S_{x_0}$; we may assume only its final vertex, say $p_0$, is outside of $S_{x_0}$. Let $\gamma'$ be the continuous plane curve produced by following $\gamma$ from $x_0$ to its last point $p_0'$ on the boundary of $S_{x_0}$ (directly before $\gamma$ touches $p_0$). We examine the possibility that $p_0'$ is on the left or right sides of $S_{x_0}$, or on the near side. 

The first case is that $p_0'$ is in the near side
\[
S_{x_0}^n = \{v \in S_{x_0} : \|v\|_2 =  (1-K_{x_0})\|x_0\|_2\}.
\]
If this holds, let $x_0' = (1-K_{x_0})x_0$, which is in $S_{x_0}^n$; we will show that $T(0,x_0')$ is abnormally large. (Here we use the definition $T(y,z) = T([y],[z])$, where $[y]$ is the point of $\mathbb{Z}^d$ with $y \in [y]+[0,1)^d$, and similarly for $z$.) 

	\begin{figure}[h]
\hbox{\hspace{5.5cm}\includegraphics[width=0.4\textwidth, trim={0cm 0cm 0cm 0cm}, clip]{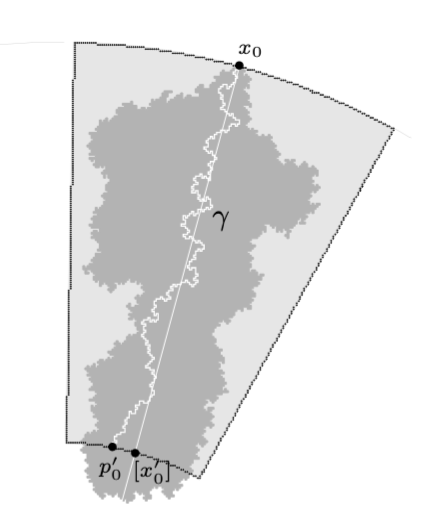}}
  \caption{The first case of the argument supposes that $\mathsf{C}$ exits the sector $S_{x_0}$ through its near side $S_{x_0}^n$. The path $\gamma$ in $\mathsf{C}$ starts at $x_0$, intersects the boundary of $S_{x_0}^n$ first at $p_0' \in S_{x_0}^n$, and it ends immediately after $p_0'$ at a vertex $p_0 \notin S_{x_0}$ (not pictured). Above, $x_0' = (1-K_{x_0})x_0$, and $[x_0']$ is the closest lattice point to $x_0'$.}
  \label{fig: fig_8}
\end{figure}

Because $p_0 \in \mathsf{C} \subset B(t_0)^c$,
\begin{equation}\label{eq: masta_marinara}
T(0,x_0') = T(0,p_0) + (T(0,x_0') - T(0,p_0)) > t_0 + (T(0,x_0') - T(0,p_0)).
\end{equation}
For large $s$, the points $x_0'$ and $p_0$ are in $3t_0\mathcal{B}$, and by occurrence of $E_2(s)$, there exists a path from $[x_0']$ to $p_0$ with $\|[x_0']-p_0\|_1$ many edges whose weights are at most $C_{20} \log t_0 \leq C_{20}\log (2c_{22}^{-1}\|x_0\|_2)$ (see \eqref{eq: t_0_relation}). This gives $T(0,x_0') - T(0,p_0) \geq -(C_{20}\log (2c_{22}^{-1}\|x_0\|_2))\|[x_0']-p_0\|_1$. However $\|[x_0']-p_0\|_1 \leq \|x_0'-p_0'\|_1 + 3 \leq \sqrt{2}\|x_0'-p_0'\|_2 + 3$, and $x_0',p_0'$ are in $S_{x_0}^n$, so if $s$ is large, then $\|x_0'-p_0'\|_2 \leq J_{x_0} \|x_0\|_2 =  (\log \|x_0\|_2)^{C_{18}-3}$. Together, for large $s$,
\begin{align*}
T(0,x_0') - T(0,p_0) &\geq - (C_{20} \log (2c_{22}^{-1}\|x_0\|_2)) (3 + \sqrt{2} (\log \|x_0\|_2)^{C_{18}-3}) \\
&\geq - (\log \|x_0\|_2)^{C_{18}-1}.
\end{align*}
Putting this in \eqref{eq: masta_marinara} gives
\begin{equation}\label{eq: masta_marinara_2}
T(0,x_0') > t_0 - (\log \|x_0\|_2)^{C_{18}-1}.
\end{equation}

To use \eqref{eq: masta_marinara_2}, we relate the left side to $T(0,x_0)$. Although $x_0$ is not in $B(t_0)$, it is the endpoint of an edge that has an endpoint in $3t_0\mathcal{B}$, so since $E_2(s)$ occurs, $T(0,x_0) \leq t_0 + C_{20}\log t_0 \leq t_0 + C_{20}\log(2c_{22}^{-1} \|x_0\|_2)$. With \eqref{eq: masta_marinara_2}, we obtain for large $s$
\begin{equation}\label{eq: masta_marinara_3}
T(0,x_0)- T(0,x_0') \leq C_{20} \log (2c_{22}^{-1} \|x_0\|_2)+(\log \|x_0\|_2)^{C_{18}-1} < 2 (\log \|x_0\|_2)^{C_{18}-1}.
\end{equation}
We now use a bound on passage time differences established in \cite[Prop.~3.7]{DHL17} under the uniform curvature assumption. The result is that for some $c_{23},C_{24},c_{25}>0$, any $z \in \mathbb{R}^d$ with $\|z\|_2=1$, and any $k,\ell \geq 0$ with $k \geq \ell$,
\begin{equation}\label{eq: DLW_bound}
\mathbb{P}(T(0,kz) - T(0,\ell z) \geq c_{23}(k-\ell)) \geq 1-C_{24} e^{-(k-\ell)^{c_{25}}}.
\end{equation}
We put $z = x_0/\|x_0\|_2$, $k = \|x_0\|_2$, and $\ell = \|x_0'\|_2 = (1-K_{x_0})\|x_0\|_2$ to produce the bound
\begin{equation}\label{eq: busemann_applied}
\mathbb{P}\left(T(0,x_0) - T(0,x_0') < c_{23} (\log \|x_0\|_2)^{C_{18}}\right) \leq C_{24} \exp\left( - (\log \|x_0\|_2)^{C_{18}c_{25}}\right).
\end{equation}
If we define the event $G(s)$ to be
\[
G(s) = \left\{ T(0,x_0) - T(0,x_0') \geq 2 (\log \|x_0\|_2)^{C_{18}-1} \text{ for all } x_0 \in \mathbb{Z}^d \text{ with } \|x_0\|_2 \geq \frac{c_{22}}{2}s \right\},
\]
then, by \eqref{eq: t_0_relation} and \eqref{eq: masta_marinara_3}, if $p_0'$ is in the near side $S_{x_0}^n$, then $G(s)^c$ must occur, and by \eqref{eq: busemann_applied}, we get
\begin{align*}
\mathbb{P}(G(s)^c) &\leq \sum_{\|x_0\|_2 \geq \frac{c_{22}}{2}s} \mathbb{P}\left(T(0,x_0) - T(0,x_0') < 2 (\log \|x_0\|_2)^{C_{18}-1}\right) \\
&\leq C_{24} \sum_{\|x_0\|_2 \geq \frac{c_{22}}{2}s} \exp\left( - ( \log \|x_0\|_2)^{C_{18}c_{25}}\right).
\end{align*}
Here we have used that for large $s$, $2(\log \|x_0\|_2)^{C_{18}-1} < c_{23}(\log \|x_0\|_2)^{C_{18}}$. Assuming $C_{18}$ is chosen larger than $c_{25}^{-1}$, we get $\mathbb{P}(G(s)^c) \to 0$ as $s \to \infty$. In summary, we can return to \eqref{eq: new_start} and write
\begin{align}
&\lim_{s \to \infty} \mathbb{P}(M(t) \geq (\log t)^{3C_{18}} \text{ for some }t \geq s) \nonumber \\
=~& \lim_{s \to \infty} \mathbb{P}\left(E_1(s) \cap E_2(s)\cap G(s) \cap \{M(t) \geq (\log t)^{3C_{18}} \text{ for some } t \geq s\}\right), \label{eq: new_start_again}
\end{align}
observing now that any outcome in the event in the probability in \eqref{eq: new_start_again} must have the property that $p_0'$ is on the union of the left and right sides of $S_{x_0}$:
\begin{equation}\label{eq: last_case}
|\theta(x_0,p_0')| = J_{x_0}.
\end{equation}

This brings us to deal with the second case, that \eqref{eq: last_case} holds in our outcome. Here, the idea is that geodesics (optimal paths in the definition of $T(x,y)$---these exist a.s.~from \cite[Thm.~4.2]{ADH17}) between some point nearby $x_0$ and the origin must avoid (``go around'') the component $\mathsf{C}$, and therefore deviate significantly from the straight line connecting the point to the origin. This is unlikely due to geodesic wandering estimates from \cite{N95}. 

Our two possible ``nearby'' points are $y_0,z_0 \in \mathbb{R}^2$, defined to have $\|y_0\|_2 = \|z_0\|_2 = (1+K_{x_0})\|x_0\|_2$, $\theta(y_0,x_0) =  J_{x_0}/2$, and $\theta(z_0,x_0) = -J_{x_0}/2$. Let $A_{x_0}^{(i)}, i=1,2$ be defined as follows.
\begin{enumerate}
\item $A_{x_0}^{(1)}$ is the event that some geodesic from $[y_0]$ to 0 has a point $x \in \mathbb{R}^2$ with $\|x\|_2 \geq (1-K_{x_0})\|x_0\|_2$ and $\theta(x,x_0) = 0$ or $J_{x_0}$.
\item $A_{x_0}^{(2)}$ is the event that some geodesic from $[z_0]$ to 0 has a point $x \in \mathbb{R}^2$ with $\|x\|_2 \geq (1-K_{x_0})\|x_0\|_2$ and $\theta(x,x_0) = -J_{x_0}$ or $0$. 
\end{enumerate}

	\begin{figure}[h]
\hbox{\hspace{5cm}\includegraphics[width=0.4\textwidth, trim={0cm 0cm 0cm 0cm}, clip]{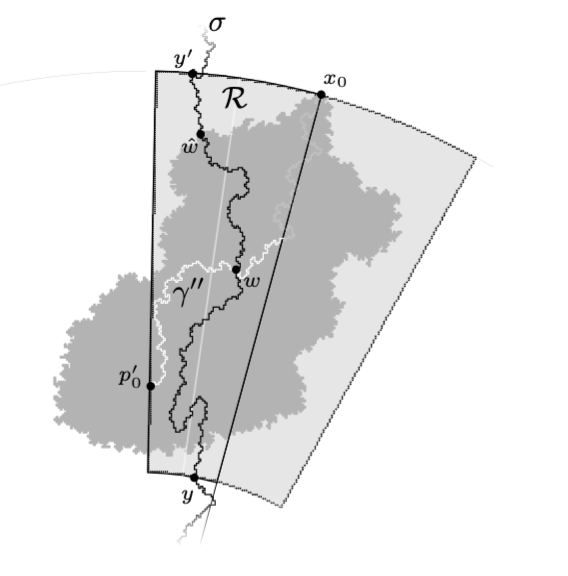}}
  \caption{In the figure, the region $\mathcal{R}$ is the left half of the sector $S_{x_0}$. In the second case considered for the argument, $\mathsf{C}$ exits $S_{x_0}$ through one of its sides, the side of $\mathcal{R}$ above. The path $\gamma$ plays an analogous role to the first case of the argument, excepting that $p_0'$ is no longer on the near boundary of $S_{x_0}$, and contains a subpath $\gamma''$ spanning opposite sides of $\mathcal{R}$. Under the event $A_{x_0}^{(1)}$, planarity forces a geodesic $\sigma$ joining $[y_0]$ (not pictured) to the origin to cross $\gamma''$ at a vertex $w$. The path $\sigma$ is further decomposed at the last point $z'$ on $\sigma$ with $\|z'\|_2 = \|x_0\|_2$ and the first point $z$ on $\sigma$ with $\|z\|_2 = (1-K_{x_0})\|x_0\|_2$, denoted $y'$ and $y$ respectively.}
  \label{fig: fig_9}
\end{figure}

We claim that because \eqref{eq: last_case} holds, 
\begin{equation}\label{eq: at_least_one}
\text{at least one of }A_{x_0}^{(1)}\text{ or }A_{x_0}^{(2)} \text{ occurs.}
\end{equation} 
To see why, let us assume first that $\theta(p_0',x_0) = J_{x_0}$. Then $\gamma'$, which we defined in the paragraph following \eqref{eq: M_t_0_pre_bound}, contains a segment $\gamma''$ which crosses the region 
\[
\mathcal{R} = \left\{v \in \mathbb{R}^2 : \frac{\|v\|_2}{\|x_0\|_2} \in  [(1-K_{x_0}),1], ~\theta(v,x_0) \in [0,J_{x_0}]\right\}
\]
between its two side boundaries; see Fig.~\ref{fig: fig_9}. This is because $\gamma'$ cannot exit $S_{x_0}$ through the far or near boundaries. Assume for a contradiction that $A_{x_0}^{(1)}$ does not occur, and let $\sigma$ be any geodesic from $[y_0]$ to $0$. Observe that for large $s$, we have $\theta([y_0],x_0) \in (0,J_{x_0})$. The segment of $\sigma$ from $[y_0]$ to its first point $y$ with $\|y\|_2 = (1-K_{x_0})\|x_0\|_2$ cannot contain any $x$ with $\theta(x,x_0) = 0$ or $J_{x_0}$, so it must contain a segment $\sigma'$ of $\sigma$ (starting at its last point $y'$ with $\|y'\|_2 = \|x_0\|_2$ before $y$ and ending at $y$) that crosses $\mathcal{R}$ from its far boundary to its near boundary. By planarity, $\sigma'$ must intersect $\gamma'$, and they must intersect at a vertex $w$.
%
%
%
%
We know $w \in \mathsf{C}$, so $T(0,w) > t_0$. Furthermore, $T(0,y_0) \geq T(0,w)$, so $[y_0] \notin B(t_0)$. But $[y_0] \notin \mathsf{C}$ by maximality of $x_0$, so $[y_0]$ is in a different component of $B(t_0)^c$. Starting from $[y_0]$, the geodesic $\sigma$ must therefore touch some $\hat w \in B(t_0)$ before it reaches $w$. This gives a contradiction because then $t_0 \geq T(0,\hat w) \geq T(0,w) > t_0$. We conclude that $A_{x_0}^{(1)}$ occurs. If we suppose that $\theta(p_0',x_0) = - J_{x_0}$ instead, a similar argument shows that $A_{x_0}^{(2)}$ occurs.

Returning to \eqref{eq: new_start_again}, the last paragraph plus a union bound gives
\[
\lim_{s \to \infty} \mathbb{P}(M(t) \geq (\log t)^{3C_{18}} \text{ for some }t \geq s) \leq \lim_{s \to \infty} \sum_{x_0 \in \mathbb{Z}^2 : \|x_0\|_2 \geq \frac{c_{22}}{2}s} \mathbb{P}(A_{x_0}^{(1)} \cup A_{x_0}^{(2)}). 
\]
To complete the proof of Theorem~\ref{thm: curvature_upper_bound}, we will show that this limit is zero, and to do this, we will prove that
\begin{equation}\label{eq: final_bound}
\sum_{x_0 \in \mathbb{Z}^2} \mathbb{P}(A_{x_0}^{(1)}) < \infty.
\end{equation}
A symmetric argument will establish the same bound for the sum of $\mathbb{P}(A_{x_0}^{(2)})$, and this will finish the proof.

Assertion \eqref{eq: final_bound} will follow from a lemma that summarizes some estimates from \cite{N95}. If $x,y \in \mathbb{Z}^d$, we write
\[
\text{out}(y,x) = \{z \in \mathbb{Z}^d : T(y,z) = T(y,x) + T(x,z)\}
\]
for the set of vertices $z$ such that a geodesic from $y$ to $z$ goes through $x$. The lemma states that with high probability, vertices in $\text{out}(0,x)$ have small angle from $x$. In \cite{N95}, this is used to show that the origin has a ``$r^{-1/4}$-straight geodesic tree.'' The argument from \cite{N95} assumes that the distribution of $\tau_e$ is continuous, but this is not needed. Only the uniform curvature assumption is required. Recall Definition~\ref{def: uniform_curvature}, which introduces the number $\eta$.
\begin{Lemma}\label{lem: geo_straight}
Let $p \in (0,1/(2\eta))$. There exist $C_{26},c_{27}>0$ such that for any $r\geq 1$,
\[
\mathbb{P}(|\theta(x,z)| \leq C_{26}\|x\|_2^{-p} \text{ for all } z \in \text{out}(0,x) \text{ and } x \text{ with } \|x\|_2 \geq r) \geq 1 - C_{26}\exp\left(-r^{c_{27}}\right).
\]
\end{Lemma}
\begin{proof}
The proof is nearly the same as that of \cite[Prop.~3.2]{N95}, so we omit some details. For a vertex $x \neq 0$, let $C_x$ be the sector portion 
\[
C_x = \{z \in \mathbb{Z}^d : g(z) \in [g(x) - g(x)^{1-\eta p}, 2 g(x)], |\theta(z,x)| \leq g(x)^{-p}\}.
\]
The vertices in the boundary set $\{y \in C_x^c : \exists z \in C_x \text{ such that } \|z-y\|_1 = 1\}$ split into three sets: $\partial_i C_x$ is those $y$ with $g(y) < g(x)-g(x)^{1-\eta p}$, $\partial_o C_x$ is those $y$ with $g(y) > 2g(x)$, and $\partial_s C_x$ is those $y$ with $\theta(x,y) > g(x)^{-p}$. Let $G_x$ be the event $\{\text{out}(0,x) \cap (\partial_i C_x \cup \partial_s C_x) \neq \emptyset\}$. (The function $g$ was defined below \eqref{eq: shape_theorem}.) Then the argument leading to \cite[Eq.~(3.3)]{N95} gives that for some $C_{28},c_{29}$, we have $\mathbb{P}(G_x) \leq C_{28}\|x\|_2^d \exp\left( - c_{29}\|x\|_2^{1/2-\eta p}\right)$. (The only difference is that \cite{N95} takes $\eta = 2$ but we have general $\eta$.) By a union bound, if $c_{30} < 1/2 - \eta p$,
\begin{equation}\label{eq: G_x_bound}
\mathbb{P}(G_x \text{ occurs for some } x \in \mathbb{Z}^d \text{ with } \|x\|_2 \geq r) \leq C_{31}\exp\left( - r^{c_{30}}\right).
\end{equation}

Fix an outcome in the event $\cap_{\|x\|_2 \geq r} G_x^c$ and let $x \in \mathbb{Z}^d$ with $\|x\|_2 \geq r$. If $z \in \text{out}(0,x)$, consider a geodesic from $0$ to $z$ that contains $x$, and define a sequence of points inductively by $x_0=x$, and for $i \geq 1$, $x_i$ is the first point of the geodesic after $x_{i-1}$ that lies in $\partial_o C_{x_{i-1}}$. (It must touch this set and not the set $\partial_i C_{x_{i-1}} \cup \partial_s C_{x_{i-1}}$ if it leaves $C_{x_{i-1}}$ because $G_{x_{i-1}}^c$ occurs.) If such a point does not exist for a particular $i = I$, because the geodesic does not leave $C_{x_{i-1}}$ before touching $z$, we set $x_I = x_{I+1} = \dots = z$. Because $x_i$ is adjacent to $C_{x_{i-1}}$, we have $|\theta(x_i,x_{i-1})| \leq C_{32} / \|x_{i-1}\|_2^p$ , and so 
\[
|\theta(z,x)| \leq \sum_{i=1}^\infty |\theta(x_i,x_{i-1})| \leq C_{32} \sum_{i=1}^I \|x_{i-1}\|_2^{-p}.
\]
However, for $i=1, \dots, I-1$, we have $\|x_{i-1}\|_2 \geq C_{33}^{i-1} \|x\|_2$ for some $C_{33}>1$, since $x_{i-1} \in \partial_o C_{x_{i-2}}$. Therefore $|\theta(z,x)| \leq C_{32}\|x\|_2^{-p} \sum_{i=1}^\infty C_{33}^{-p(i-1)}$. In other words, for $C_{26} = C_{32}\sum_{i=1}^\infty C_{33}^{-p(i-1)}$, any outcome in $\cap_{\|x\|_2 \geq r} G_x^c$ has $|\theta(z,x)| \leq C_{26}\|x\|_2^{-p}$ so long as $\|x\|_2 \geq r$ and $z \in \text{out}(0,x)$. The estimate \eqref{eq: G_x_bound} finishes the proof.
\end{proof}

Using Lemma~\ref{lem: geo_straight}, we can show \eqref{eq: final_bound}, and therefore finish the proof of Theorem~\ref{thm: curvature_upper_bound}. Suppose that $A_{x_0}^{(1)}$ occurs. Choose a point $x \in \mathbb{R}^2$ such that $\|x\|_2 \geq (1-K_{x_0})\|x_0\|_2$ and $\theta(x,x_0) = 0$ or $J_{x_0}$, but that $x$ is on a geodesic from $[y_0]$ to 0. Let $x'$ be a vertex on this geodesic such that $\|x-x'\|_1 \leq 1$. The law of sines from trigonometry implies that if $\theta_{[y_0]}(0,x')$ is the angle between $0$ and $x'$ as measured from $[y_0]$, then
\begin{equation}\label{eq: law_of_sines}
\|x'\|_2 \sin |\theta([y_0],x')| = \|[y_0]-x'\|_2 \sin |\theta_{[y_0]}(0,x')|.
\end{equation}
To estimate these quantities, we observe first that for large $\|x_0\|_2$, we have 
\begin{equation}\label{eq: x_bound}
\|x'\|_2 \geq \|x_0\|_2/2.
\end{equation}
Next, because $|\theta(x,y_0)| = J_{x_0}/2$, we have 
\begin{equation}\label{eq: angle_lower_bound}
|\theta([y_0],x')| \in \left( \frac{J_{x_0}}{3}, 2 \frac{J_{x_0}}{3} \right)
\end{equation}
so long as $\|x_0\|_2$ is large enough. In particular, if $\|x_0\|_2$ is large, then $|\theta([y_0],x')|$ is small, and so
\begin{equation}\label{eq: sin_angle_lower_bound}
\sin|\theta([y_0],x)| \geq \frac{|\theta([y_0],x')|}{2} \geq \frac{J_{x_0}}{6}.
\end{equation}


The term $\sin |\theta_{[y_0]}(0,x')|$ can be bounded using Lemma~\ref{lem: geo_straight}. For $u \in \mathbb{Z}^d$ and $p \in (0,1/(2\eta))$ fixed, write $F_u(r)$ for the event described in Lem.~\ref{lem: geo_straight}, translated in the natural way so that the origin is mapped to $u$. Precisely, if $T_u$ is the translation of $\mathbb{R}^d$ such that $T_u(0) = u$, then $F_u(r)$ is the event that the image configuration $(\tau_{T_u^{-1}(e)})$ is in the event described in Lem.~\ref{lem: geo_straight}. We observe that $\|[y_0]-x'\|_2 \geq \|y_0-x\|_2 - \|[y_0]-y_0\|_2 - \|x-x'\|_2 \geq J_{x_0} \|x_0\|_2/2 - 3$, so if $\|x_0\|_2$ is large and $F_{[y_0]}(r)$ occurs for $r = J_{x_0}\|x_0\|_2/3$, then we must have 
\begin{equation}\label{eq: newman_bound}
|\theta_{[y_0]}(0,x')| \leq C_{26}\|[y_0]-x'\|_2^{-p}.
\end{equation}
Putting this, \eqref{eq: x_bound}, and \eqref{eq: sin_angle_lower_bound} into \eqref{eq: law_of_sines} produces
\begin{equation}\label{eq: another_almost_end_angle}
\frac{1}{12}(\log \|x_0\|_2)^{C_{18}-3} = \frac{\|x_0\|_2 J_{x_0}}{12} \leq \|[y_0]-x'\|_2 \sin |\theta_{[y_0]}(0,x')| \leq C_{26} \|[y_0]-x'\|_2^{1-p}.
\end{equation}
For large $\|x_0\|_2$, we conclude
\begin{equation}\label{eq: almost_end_angle}
\|y_0-x\|_2 \geq c_{34} (\log \|x_0\|_2)^{\frac{C_{18}-3}{1-p}}.
\end{equation}

To proceed from \eqref{eq: almost_end_angle}, we assume for a contradiction that $F_{[y_0]}(r)$ occurs (so that \eqref{eq: almost_end_angle} holds) and consider two cases. If $\|x\|_2 \leq \|y_0\|_2$, then $\|x\|_2/\|x_0\|_2 \in [1-K_{x_0},1+K_{x_0}]$. If $\theta(x,x_0) = 0$, then 
\[
\|y_0-x\|_2 \leq \|y_0-x_0\|_2 + \|x-x_0\|_2 \leq \left(\frac{J_{x_0}}{2}+K_{x_0}\right)\|x_0\|_2 + K_{x_0}\|x_0\|_2 \leq 3K_{x_0}\|x_0\|_2.
\]
By symmetry, the inequality $\|y_0-x\|_2 \leq 3K_{x_0}\|x_0\|_2$ also holds if $\theta(x,x_0) = J_{x_0}$. Putting it in \eqref{eq: almost_end_angle}, we find
\[
3(\log \|x_0\|_2)^{C_{18}} \geq c_{34} (\log \|x_0\|_2)^{\frac{C_{18}-3}{1-p}},
\]
which is false if $C_{18}>3/p$ and $\|x_0\|_2$ is large. Otherwise, if $\|x\|_2 \geq \|y_0\|_2$, then 
\begin{equation}\label{eq: last_x_bound}
\|x'\|_2 \geq \|[y_0]\|_2 - 1 - \sqrt{2} \geq \|[y_0]\|_2 - 3.
\end{equation} 
Because $|\theta_{[y_0]}(0,x')| + |\theta(x',[y_0])| + |\theta_{x'}(0,[y_0])| = \pi$, we see for large $\|x_0\|_2$ from \eqref{eq: angle_lower_bound} and \eqref{eq: newman_bound} that $|\theta_{x'}(0,[y_0])| \geq 3\pi/4$, and so $\cos | \theta_{x'}(0,[y_0])| \leq -1/\sqrt{2}$. The law of cosines along with \eqref{eq: another_almost_end_angle} and \eqref{eq: last_x_bound} then gives for large $\|x_0\|_2$
\begin{align*}
\|[y_0]\|_2^2 &= \|x'\|_2^2 + \|[y_0]-x'\|_2^2 - 2\|x'\|_2\|[y_0]-x'\|_2 \cos \theta_{x'}(0,[y_0]) \\
&\geq \|x'\|_2^2 + \sqrt{2} \|x'\|_2 \|[y_0]-x'\|_2 \\
&\geq \|x'\|_2^2 + 7 \|x'\|_2 \\
&\geq (\|[y_0]\|_2-3)^2 + 7\|[y_0]\|_2 - 21.
\end{align*}
This is a contradiction if $\|x_0\|_2$ is large.

We conclude that if $\|x_0\|_2$ is sufficiently large, then $A_{x_0}^{(1)}\subset F_{[y_0]}(r)^c$ with $r = J_{x_0}\|x_0\|_2/3$. Lem.~\ref{lem: geo_straight} gives the bound
\[
\mathbb{P}(A_{x_0}^{(1)}) \leq \mathbb{P}(F_{[y_0]}(r)^c) \leq C_{26}\exp\left( -  (J_{x_0}\|x_0\|_2/3)^{c_{27}}\right).
\]
This is summable over $x_0 \in \mathbb{Z}^2$ so long as $C_{18} > 3 + c_{27}^{-1}$. This completes the proof of \eqref{eq: final_bound}.

\section{Proof of Thm.~\ref{thm: no_curvature_upper_bound}} \label{sec: no_curvature_upper_bound}

The proof of Thm.~\ref{thm: no_curvature_upper_bound} is like that of Thm.~\ref{thm: curvature_upper_bound}, and will use similar constructions, so we give fewer details and focus on the modifications needed to apply the argument. There are two main differences. First, instead of using the bound \eqref{eq: DLW_bound} on passage time differences (which requires the uniform curvature assumption), we will use a general concentration inequality. Second, instead of using Lem.~\ref{lem: geo_straight} on the straightness of geodesics (also requiring curvature), we will use Kesten's lemma. 

The concentration inequality states the there exists $C_{35}>0$ such that for all large $x \in \mathbb{Z}^d$,
\begin{equation}\label{eq: general_concentration}
\mathbb{P}\left( |T(0,x) - g(x)| \geq C_{35}\sqrt{g(x) \log g(x)}\right) \leq \|x\|_1^{-100}.
\end{equation}
This inequality follows from standard results. First, it suffices to prove it with $\sqrt{g(x) \log g(x)}$ replaced by $\sqrt{\|x\|_1 \log \|x\|_1}$. In this form, it follows from \cite[Prop.~1.1]{DK16}, which says that for some $C_{36} > 0$, we have $0 \leq \mathbb{E}T(0,x) - g(x) \leq C_{36} \sqrt{\|x\|_1 \log \|x\|_1}$, and \cite[Thm.~1.1]{DHS14}, which says that $\mathbb{P}(|T(0,x) - \mathbb{E}T(0,x)| \geq \lambda \sqrt{\|x\|_1/\log \|x\|_1}) \leq e^{-c_{37}\lambda}$ for some constant $c_{37}>0$ and all $\lambda \geq 0$ and nonzero $x \in \mathbb{Z}^d$. From these two, we just have to choose $\lambda = 2C_{35} \log \|x\|_1$ for large enough $C_{35}$.

The second tool, Kesten's lemma \cite[Prop.~5.8]{aspects}, states that there exist $a,c_{38}>0$ such that
\begin{equation}\label{eq: kestens_lemma}
\mathbb{P}(\exists \text{ edge-self-avoiding path } \gamma \text{ containing } 0 \text{ with } \#\gamma \geq n \text{ but } T(\gamma) < an) \leq e^{-c_{38}n}.
\end{equation}
Here, $\#\gamma$ is the number of edges in $\gamma$. This result will allow us to show in \eqref{eq: kesten_application} that if a geodesic deviates too far from a straight line, it must have a long segment with high passage time.

As in the proof of Thm.~\ref{thm: curvature_upper_bound}, we define events $E_i(s)$ for $s \geq 0$ and a constant $C_{39}>0$ as
\begin{align*}
E_1(s) &= \left\{ \frac{1}{2} \mathcal{B} \subset \frac{1}{t} \widetilde{B}(t) \subset 2 \mathcal{B} \text{ for all } t \geq s \right\} \\
E_2(s) &= \{ \tau_e \leq C_{39} \log t \text{ for all }e \text{ with an endpoint in }3t \mathcal{B} \text{ and all } t \geq s\} \\
E_3(s) &= \left\{ |T(0,x) - g(x)| \leq C_{35} \sqrt{g(x) \log g(x)} \text{ for all integer }x \in 3t\mathcal{B} \setminus ((t/3)\mathcal{B}) \text{ and all } t \geq s\right\}.
\end{align*}
As in \eqref{eq: new_start}, for some $C_{39}$ large enough, and any $C_{40}>0$,
\begin{align}
&\lim_{s \to \infty} \mathbb{P}\left(M(t) \geq C_{40} t \log t \text{ for some }t \geq s\right) \nonumber \\
=~& \lim_{s \to \infty} \mathbb{P}\left(E_1(s) \cap E_2(s) \cap \left\{M(t) \geq C_{40}t \log t \text{ for some } t \geq s\right\}\right). \label{eq: again_new_start}
\end{align}
Using \eqref{eq: general_concentration} with a union bound, we obtain
\begin{equation}\label{eq: E_3_bound}
\mathbb{P}(E_3(s)^c) \leq \sum_{x \in ((s/3)\mathcal{B})^c} \|x\|_1^{-100} \to 0 \text{ as } s \to \infty.
\end{equation}
Last, we let $E_4(s)$ be the event that, for all $t \geq s$, and all vertices $x \in 3t\mathcal{B} \setminus ((t/3)\mathcal{B})$, any edge-self-avoiding path $\Gamma$ containing $x$ with at least $(12C_{35}/a) \sqrt{g(x) \log g(x)}$ many edges satisfies $T(\Gamma) \geq 12C_{35} \sqrt{g(x) \log g(x)}$. To prove that
\begin{equation}\label{eq: E_4_prob}
\mathbb{P}(E_4(s)^c) \to 0 \text{ as } s \to \infty,
\end{equation}
we use \eqref{eq: kestens_lemma} with a union bound. We obtain
\[
\mathbb{P}(E_4(s)^c) \leq \sum_{x \in \left( \frac{s}{3} \mathcal{B}\right)^c} e^{-c_{38} 12C_{35} a^{-1} \sqrt{g(x)\log g(x)}} \to 0 \text{ as } s \to \infty,
\]
This shows \eqref{eq: E_4_prob}. Putting \eqref{eq: E_3_bound} and \eqref{eq: E_4_prob} into \eqref{eq: again_new_start}, we find
\begin{align}
&\lim_{s \to \infty} \mathbb{P}\left(M(t) \geq C_{40}t \log t \text{ for some }t \geq s\right) \nonumber \\
=~& \lim_{s \to \infty} \mathbb{P}\left(\left( \cap_{i=1}^4 E_i(s) \right) \cap \{M(t) \geq C_{40}t \log t \text{ for some } t \geq s\}\right). \label{eq: again_new_start2}
\end{align}

The rest of the proof serves to show that if $C_{40}$ is large, then \eqref{eq: again_new_start2} is zero. To do this, we choose an outcome in the event in \eqref{eq: again_new_start2}, and let $t_0 \geq s$. Pick $x_0$ as any vertex with maximal value of $g(x_0)$ in a bounded component $\mathsf{C}$ of $B(t_0)^c$ with the largest number of vertices. Analogously to \eqref{eq: S_x_0_def}, let
\[
S_{x_0} = \left\{ v \in \mathbb{R}^2 : |\theta(v,x_0)| \leq J_{x_0} \text{ and } 1 - K_{x_0} \leq \frac{g(v)}{g(x_0)} \leq 1 \right\},
\]
where
\[
K_{x_0} = \frac{3C_{35} \sqrt{g(x_0) \log g(x_0)}}{g(x_0)}
\]
and 
\[
J_{x_0} = \frac{64}{a c_{22}} K_{x_0}.
\]
By a similar argument to that which gave \eqref{eq: containment}, if $s$ is large, because $M(t_0) \geq C_{40}t_0 \log t_0$,
\[
\mathsf{C} \text{ contains a vertex in } S_{x_0}^c,
\]
so long as $C_{40}$ is fixed to be large enough. Because of this, we can find a path $\gamma$ contained in $\mathsf{C}$ starting at $x_0$ that ends at a vertex outside of $S_{x_0}$; we may assume only its final vertex, say $p_0$, is outside of $S_{x_0}$. We also let $\gamma'$ be the continuous plane curve produced by following $\gamma$ from $x_0$ to its last point $p_0'$ on the boundary of $S_{x_0}$ (directly before $\gamma$ touches $p_0$). As we have done in the last section, we must exclude the possibility that $p_0'$ is on the left or right sides of $S_{x_0}$, or on the near side. The point $p_0'$ cannot be on the far side only because $g(x_0)$ is maximal among vertices in $\mathsf{C}$.

The first case is that $p_0'$ is in the near side
\[
S_{x_0}^n = \{v \in S_{x_0} : g(v) = (1-K_{x_0})g(x_0)\}.
\]
If $s$ is large, then $p_0 \in 3t_0 \mathcal{B} \setminus ((t_0/3)\mathcal{B})$, so since $E_3(s)$ occurs, we have for some $C_{41}>0$
\begin{align*}
T(0,p_0) &\leq g(p_0) + C_{35}\sqrt{g(p_0) \log g(p_0)} \\
&\leq C_{41} + g(p_0') + C_{35}\sqrt{g(p_0') \log g(p_0')} \\
&= C_{41} + g(x_0) - 3C_{35} \sqrt{g(x_0) \log g(x_0)} + C_{35}\sqrt{g(p_0') \log g(p_0')} \\
&\leq g(x_0) - 2C_{35}\sqrt{g(x_0)\log g(x_0)}.
\end{align*}
Because $T(0,x_0) \geq g(x_0) - C_{35} \sqrt{g(x_0) \log g(x_0)}$, we obtain
\begin{equation}\label{eq: lower_to_contradict}
T(0,x_0) - T(0,p_0) \geq C_{35} \sqrt{g(x_0) \log g(x_0)}
\end{equation}
as long as $s$ is large. On the other hand, $p_0 \notin B(t_0)$, so $T(0,p_0) > t_0$. Furthermore, $x_0$ is an endpoint of an edge with an endpoint in $B(t_0)$, and this edge must have weight at most $C_{39}\log t_0$ because $E_2(s)$ occurs. Therefore
\[
T(0,x_0) - T(0,p_0) \leq C_{39}\log t_0 + t_0 - t_0 = C_{39} \log t_0.
\]
Because $t_0 \leq (2/c_{22}) \|x_0\|_2$ from \eqref{eq: t_0_relation}, this contradicts \eqref{eq: lower_to_contradict}.

The second case is that $p_0'$ satisfies $|\theta(x_0,p_0')| = J_{x_0}$. We will suppose that $\theta(x_0,p_0') = J_{x_0}$, as the other possibility is dealt with using a similar argument. Let $y_0 \in \mathbb{R}^2$ satisfy $\theta(y_0,x_0) = J_{x_0}/2$ and $g(y_0) = g(x_0)$, and choose a vertex $\bar{y}_0$ with $g(\bar{y}_0) > g(y_0)$ but $\|y_0 - \bar{y}_0\|_1 = 1$. Let $\sigma$ be any geodesic from $\bar{y}_0$ to 0. As in the proof of \eqref{eq: at_least_one}, as $\sigma$ proceeds from $\bar{y}_0$ to 0, planarity implies it must touch one of the rays
\[
R = \{v \in \mathbb{R}^2 : \theta(v,x_0) = 0\} \text{ or } R' = \{v \in \mathbb{R}^2 : \theta(v,x_0) = J_{x_0}\}
\]
before touching the set $B' = \{v : g(v) = (1-K_{x_0})g(x_0)\}$. Indeed, if this were false, then because the curve $\gamma'$ connecting $x_0$ to $p_0'$ must contain a segment crossing the region
\[
\mathcal{R} = \left\{ v \in \mathbb{R}^2 : \frac{g(v)}{g(x_0)} \in [(1-K_{x_0}),1], \theta(v,x_0) \in [0,J_{x_0}]\right\}
\]
between its two side boundaries, $\sigma$ would have to intersect $\gamma'$ at a vertex $w$. As in the last section, this gives a contradiction because $w \in \mathsf{C}$, so $T(0,w) > t_0$, but because $\sigma$ originates outside of $\mathsf{C}$, it must touch some $\hat{w} \in B(t_0)$ before reaching $w$, and so $t_0 \geq T(0,\hat{w}) \geq T(0,w) > t_0$.

Without loss of generality, we suppose that $\sigma$ touches some $p_1 \in R'$ before some $p_2 \in B'$. Let $\bar{p}_1$ be the vertex we encounter on $\sigma$ directly before $p_1$ as we proceed from $\bar{y}_0$ to 0, and let $\bar{p}_2$ be the vertex we encounter on $\sigma$ directly after $p_2$. Because $p_2 \in B'$, we have $g(p_2) = g(x_0) - 3 C_{35} \sqrt{g(x_0)\log g(x_0)}$. The event $E_2(s) \cap E_3(s)$ occurs, so for large $s$,
\begin{align*}
T(0,\bar{p}_2) &\geq g(\bar{p}_2) - C_{35}\sqrt{g(\bar{p}_2)\log g(\bar{p}_2)} \\
&\geq g(p_2) - C_{35}\sqrt{g(p_2) \log g(p_2)} - C_{42} \\
&= g(x_0) - 3C_{35} \sqrt{g(x_0)\log g(x_0)} - C_{35} \sqrt{g(p_2)\log g(p_2)} - C_{42} \\
&\geq g(x_0) - 4C_{35} \sqrt{g(x_0) \log g(x_0)}.
\end{align*}
Here, $C_{42}>0$ is a constant. Because $\bar{p}_1$ appears first on $\sigma$, we have $T(0,\bar{p}_1) \geq T(0,\bar{p}_2)$, so
\begin{equation}\label{eq: p_1_lower}
T(0,\bar{p}_1) \geq g(x_0) - 4C_{35} \sqrt{g(x_0) \log g(x_0)}.
\end{equation}

To obtain an upper bound on $T(0,\bar{p}_1)$, we use the occurrence of $E_2(s)\cap E_3(s)$ to estimate
\begin{align}
T(0,\bar{p}_1) = T(0,\bar{y}_0) - T(\bar{y}_0,\bar{p}_1) &\leq T(0,y_0) + C_{39}\log t_0 - T(\bar{y}_0, \bar{p}_1) \nonumber \\
&\leq g(x_0) + C_{35}\sqrt{g(x_0)\log g(x_0)} + C_{39} \log t_0 - T(\bar{y}_0, \bar{p}_1). \label{eq: T_bar_upper_bound}
\end{align}
Any path from $\bar{y}_0$ to $\bar{p}_1$ must have at least $\|\bar{y}_0 - \bar{p}_1\|_1$ many edges, and by \eqref{eq: balls_def}, if $s$ is large,
\begin{align*}
\|\bar{y}_0 - \bar{p}_1\|_1 \geq \|y_0-p_1\|_2 - 2 &\geq \sin\left( \frac{J_{x_0}}{2}\right) \|y_0\|_2 - 2 \\
&\geq  \frac{J_{x_0}}{4} \|y_0\|_2 - 2 \\
&= \frac{3}{4}\cdot \frac{64}{a c_{22}}C_{35}\sqrt{g(x_0)\log g(x_0)} \frac{\|y_0\|_2}{g(y_0)} - 2 \\
&\geq \frac{3}{8} c_{22} \cdot \frac{64}{ac_{22}}C_{35} \sqrt{g(x_0)\log g(x_0)}.
\end{align*}
If $s$ is large, this is bigger than $(12C_{35}/a)\sqrt{g(\bar{y}_0) \log g(\bar{y}_0)}$, so since $E_4(s)$ occurs, 
\begin{equation}\label{eq: kesten_application}
T(\bar{y}_0,\bar{p}_1) \geq 12C_{35} \sqrt{g(\bar{y}_0) \log g(\bar{y}_0)} \geq 6C_{35}\sqrt{g(x_0) \log g(x_0)}.
\end{equation}
Returning to \eqref{eq: T_bar_upper_bound}, for large $s$, we get
\[
T(0,\bar{p}_1) \leq g(x_0) + \left(C_{35} - 6C_{35}\right) \sqrt{g(x_0) \log g(x_0)} + C_{39}\log t_0.
\]
This contradicts \eqref{eq: p_1_lower} for large $s$, since $t_0 \leq (2/c_{22})\|x_0\|_2$ from \eqref{eq: t_0_relation}.


\begin{thebibliography}{1}

\bibitem{ADH17}
Auffinger, A.; Damron, M.; Hanson, J. 50 years of first-passage percolation. University Lecture Series, 68. {\it American Mathematical Society, Providence, RI,} 2017. v+161 pp.

\bibitem{BBR10}
Barequet, R.; Barequet, G.; Rote, G. Formulae and growth rates of high-dimensional polycubes. {\it Combinatorica}. {\bf 30} (2010), 257--275.

\bibitem{B15}
Bouch, G. The expected perimeter in Eden and related growth processes. {\it J. Math. Phys.} {\bf 56}

\bibitem{CT16}
Cerf, R.; Th\'eret, M. Weak shape theorem in first passage percolation with infinite passage times. {\it Ann. Inst. H. Poincar\'e (B) Probab. et Statist.} {\bf 52} (2016), 1351--1381.

\bibitem{DHL17}
Damron, M.; Hanson, J.; Lam, W.-K. The size of the boundary in first-passage percolation. {\it Ann. Appl. Probab.} {\bf 28} (2018), 3184--3214.

\bibitem{DHS14}
Damron, M.; Hanson, J.; Sosoe, P. Subdiffusive concentration in first passage percolation. {\it Electron. J. Probab.} {\bf 19} (2014), 1--27.

\bibitem{DK16}
Damron, M.; Kubota, N. Rate of convergence in first-passage percolation under low moments. {\it Stoch. Proc. Appl.} {\bf 126} (2016), 3065--3076.

\bibitem{DLW17}
Damron, M.; Lam, W.-K.; Wang, X. Asymptotics for $2D$ critical first-passage percolation. {\it Ann. Probab.} {\bf 45} (2017), 2941--2970.

\bibitem{DRS16}
Damron, M.; Rassoul-Agha, F.; Sepp\"al\"ainen, T. {\it Random growth models.} Proc. Sympos. Appl. Math., 75, {\it Amer. Math. Soc., Providence, RI,} 2018.

\bibitem{aspects}
Kesten, H. Aspects of first passage percolation. {\it \'Ecole d'\'et\'e de probabilit\'es de Saint-Flour, XIV - 1984}, 125--264, Lecture Notes in Math., 1180, {\it Springer, Berlin}, 1986.

\bibitem{L85}
Leyvraz, F. The ``active perimeter'' in cluster growth models: a rigorous bound. {\it J. Phys. A.} {\bf 18}, L941--L945.

\bibitem{MRS20}
Manin, F.; Rold\'an, \'E; Schweinhart, B. Topology and local geometry of the Eden model. (2020), preprint.

\bibitem{N95}
Newman, C. M. A surface view of first-passage percolation. {\it Proceedings of the International Congress of Mathematicians, Vol. 1, 2 (Z\"urich, 1994),} 1017-1023, {\it Birkh\"auser, Basel.} 

\bibitem{T13}
Tim\'ar, A. Boundary connectivity via graph theory. {\it Proc. Amer. Math. Soc.} {\bf 141} (2013), 475--480.



\end{thebibliography}
\end{document}